\documentclass[11pt]{article}
\usepackage{amsmath,amsfonts,amsthm,amscd,amssymb,graphicx}
\usepackage{subfigure}

\numberwithin{equation}{section}

\usepackage{hyperref}
\usepackage{verbatim} 
\usepackage{color}
\newtheorem{theorem}{Theorem}[section]

\newtheorem{lemma}[theorem]{Lemma}

\newtheorem{proposition}[theorem]{Proposition}

\newtheorem{remark}[theorem]{Remark}

\def\eps{\varepsilon }

\def\beq{\begin{equation}}
\def\eeq{\end{equation}}
\def\bb1{{1\!\!1}}
\def\R{\mbox{Re }}

\def\pt{\partial}

\def\ep{\epsilon}

\def\eps{\varepsilon}

\def\triangle{\Delta}
\def\bega{\begin{aligned}}
\def\enda{\end{aligned}}
\def\R{\mathbb{R}^2}

\def\lw{\left}
\def\rw{\right}

\def\R{\mathbb{R}}
\def\wtd{\widetilde}
\def\la{\langle}
\def\ra{\rangle}

\def\bcase{\begin{cases}}
\def\ecase{\end{cases}}
\def\al{\alpha}

\begin{document}

\title{Derivative estimates for screened Vlasov-Poisson system around Penrose-stable equilibria} 

\author{
 Trinh T. Nguyen\footnotemark[1]
}

\maketitle

\renewcommand{\thefootnote}{\fnsymbol{footnote}}

\footnotetext[1]{Department of Mathematics, Penn State University, State College, PA 16803. Emails: 
txn5114@psu.edu.}

\begin{abstract}
In this paper, we establish derivative estimates for the Vlasov-Poisson system with screening interactions around Penrose-stable equilibria on the phase space $\R^d_x\times \R_v^d$, with dimension $d\ge 3$. In particular, we establish the optimal decay estimates for higher derivatives of the density of  the perturbed system, precisely like the free transport, up to a log correction in time. This extends the recent work  \cite{T-R-HK} by Han-Kwan, Nguyen and Rousset to higher derivatives of the density. The proof makes use of several key observations from \cite{T-R-HK} on the structure of the forcing term in the linear problem, with induction arguments to classify all the terms appearing in the derivative estimates.
\end{abstract}

\maketitle

\section{Introduction}
\subsection{The system}
In this paper, we consider the screened Vlasov-Poisson system on the phase space  $(x,v)\in\R^d\times \R^d$, with the dimension $d\ge 3$:
\beq\label{VP}
\pt_t f+v\cdot\nabla_x f+E\cdot\nabla_vf=0
\eeq 
Here $f=f(t,x,v)\ge 0$ is the probability distribution of charged particles in plasma,
$$\rho(t,x)=\int_{\R^d}f(t,x,v)dv$$ is the electric charge density, and $$E(t,x)=-\nabla_x(1-\triangle_x)^{-1}(\rho-1)$$ is electric field.
This model has been investigated in physical literatures \cite{Sanderson, Chavanis,book1} and also recent mathematical works \cite{Jacob, T-R-HK}. We refer the readers to \cite{global-VP1, global-VP2, global-VP3, global-VP4,global-VP5} for global existence and regularity results. The system \eqref{VP} has $$(f,\rho,E)=(\mu(v),1,0)$$ as a steady solution for any smooth and decaying function $\mu(v)\ge 0$ with the normalized condition $$\int_{\R^d} \mu(v)dv=1.$$ 
We assume that $\mu(v)$ satisfies the Penrose stability condition ($\widehat{\nabla_v\mu}$ denotes the Fourier transform):
\beq \label{Penrose}
\inf_{\Im\tau\le 0}\inf_{\xi\in \R^d}\lw|
1-\int_0^\infty e^{-i\tau s}\frac{1}{1+|\xi|^2}i\xi\cdot\widehat{\nabla_v\mu}(s\xi)ds
\rw|\ge \kappa\qquad \text{for}\quad \kappa>0,
\eeq 
which is classical in the study of Landau damping \cite{Landau1,Landau}, where the authors justify the asymptotic stability for homogeneous equilibria that satisfies \eqref{Penrose} on the phase space $\mathbb{T}^d\times \R^d$, under analytic and Gevrey perturbations for the Vlasov-Poisson system. The stability condition \eqref{Penrose} is also natural in studying the quasineutral limit of the Vlasov-Poisson equations \cite{HK1,HK2, HK3} and the long time estimates for the Vlasov-Maxwell equations \cite{R2}.
\subsection{Previous works}
In \cite{Jacob},  Bedrossian, Masmoudi and Mouhot justify the asymptotic stability of the equilibria when $\mu(v)$ satisfies the condition \eqref{Penrose}. The proof is inspired by \cite{Landau} for Landau damping on the confined phase space $\mathbb{T}^d\times \R^d$. Using the dispersive mechanism in Fourier space to control the plasma echo resonance, the authors in \cite{Jacob} prove that the Fourier mode of the density $\hat\rho(t,\xi)$ decays like $\frac{1}{(t|\xi|)^{N-\delta}}$ if the initial perturbation is in Sobolev space of high regularity order $N$ and some $\delta\in (0,N)$. The decay is far from being optimal, as the dispersion in the physical space was not taken into account.~\\\\
In the recent work \cite{T-R-HK}, Han-Kwan, Nguyen, and Rousset revisit the asymptotic stability of equilibria that satisfy the condition \eqref{Penrose}. They obtain the decay estimates for the perturbed electric charge density $\rho(t)$ as follows:
\[
\|\rho(t)\|_{L^1}+\la t\ra \|\nabla_x \rho(t)\|_{L^1}+\la t \ra^d \|\rho(t)\|_{L^\infty}+\la t\ra^{d+1}\|\nabla_x \rho(t)\|_{L^\infty}\lesssim \eps_0\log(1+t),\qquad \la t\ra=\sqrt{t^2+1}.
\]
Unlike \cite{Jacob}, which relies on the nonlinear energy estimates, the authors in \cite{T-R-HK} use direct dispersive mechanism in the physical space, which is like the free transport up to $\log(t)$. This is achieved by a pointwise dispersive estimate, directly on the resolvent kernel for the linearized system around non-zero stable equilibria $\mu(v)$. This is followed by solving the equations by characteristic methods, inspired from the classical work of Bardos and Degond \cite{Bardos-Degond}. At the same time, the authors in \cite{T-R-HK} are able to propagate $C^1$ norm for the initial data  and thus allow more general perturbations. We note that the dispersive mechanism of the free transport operator $\pt_t+v\cdot\nabla_x$ on $\R^d\times \R^d$ is also one of the key ingredients in the classical results of Bardos and Degond \cite{Bardos-Degond} in 1985, where they study the asymptotic stability of Vlasov-Poisson around vacuum ($\mu(v)=0$). ~\\\\
Regarding the stability of vacuum (when $\mu(v)=0$) for the unscreened Vlasov-Poisson systems, we refer the readers to the work \cite{Hwang} for the extension of Bardos-Degond results for optimal decays of higher derivatives. In \cite{Smulevici}, Smulevici obtains the spatial decay by the vector field methods. In \cite{wang2018decay}, Wang justifies the stability of vacuum for Vlasov-Poisson by Fourier methods. In \cite{Choi-2D}, Choi, Ha and Lee justify the same result for 2D screened Vlasov-Poisson.
\section{Main results}
\subsection{Main theorem}
 In this paper, we will give decay estimates for higher derivatives of $\rho(t)$, under the assumption that the initial perturbation $f_0$ is small in suitable Sobolev space and for $\mu(v)$ satisfying decaying assumption: Given any $m\in \mathbb{N}$ and $M>0$, there exists $C_{m,M}>0$ such that 
 \beq\label{bound}
|\nabla_v^m \mu(v)|\le C_{m,M}\la v\ra^{-M}\quad \text{for all}\quad v\in \R^d.
 \eeq
 The equation for the perturbed quantities around the equilibria of \eqref{VP} reads 
\beq\label{Per-VP}
\begin{cases}
&\pt_t f+v\cdot\nabla_x f+E\cdot\nabla_v \mu=-E\cdot\nabla_v f,\\
&E=-\nabla_x(1-\triangle_x)^{-1}\rho,\\
&\rho=\int_{\R^d}fdv.
\end{cases}
\eeq 
Our main theorem is as follows:
 \begin{theorem}\label{main-thm}
Let $N>1$ be an integer. Let $\mu(v)\ge 0$ be sufficiently smooth and satisfy the Penrose stability condition \eqref{Penrose}, the decaying bound \eqref{bound} and the normalized condition $\int_{\R^d}\mu(v) dv=1$.
There exists $\eps_0>0$ such that for all $f_0(x,v)$ satisfying the smallness assumption
\[
\max_{0\le k\le N}\lw(\|\nabla_{x,v}^k f_0\|_{L^1_{x,v}}+\|\nabla_{x,v}^k f_0\|_{L^1_xL^\infty_v}\rw)\le \eps_0,
\]
the solution $f(t,x,v)$ to the equations \eqref{Per-VP} with initial data $f|_{t=0}=f_0(x,v)$ satisfies 
 \[
\max_{0\le k\le N}\lw(\la t\ra^k \|\nabla_x^k\rho(t)\|_{L^1}+\la t\ra ^{k+d}\|\nabla_x^k\rho(t)\|_{L^\infty}\rw)\lesssim \eps_0 \log(1+t).
 \]
 where $\rho(t,x)=\int_{\R^d}f(t,x,v)dv$. Here the norm $\|\cdot\|_{L^p_xL^q_v}$ is defined by 
 \[
 \|g\|_{L^p_xL^q_v}=\lw(\int_{\R^d}\|g(x,v)\|_{L^q_v}^pdx\rw)^{1/p}.
 \]
\end{theorem}
\subsection{Motivation}
The improved decays for higher derivatives of $\rho(t)$ can be seen from the free transport equation 
\[
\pt_t f_{\text{free}}+v\cdot\nabla_x f_{\text{free}}=0,\qquad f|_{t=0}=f_0.
\]
The solution is given by
\[
f_{\text{free}}=f_0(x-tv,v),\qquad \rho_{\text{free}}(t,x)=\int_{\R^d}f_0(x-tv,v)dv.
\]
Making the change of variables $w=x-tv$, we obtain
\beq\label{free}
\rho_{\text{free}}(t,x)=\int_{\R^d}f_0\lw(w,\frac{x-w}{t}\rw)t^{-d}dw.
\eeq 
Hence 
\[
\|\rho_{\text{free}}(t)\|_{L^\infty}\le t^{-d}\|f_0\|_{L^1_xL^\infty_v}\quad\text{and}\quad \|\rho_{\text{free}}(t)\|_{L^1}\le \|f_0\|_{L^1_xL^1_v}.
\]
Thus for the free transport equations, $\rho_{\text{free}}(t)$ decays like $t^{-d}$ in $L^\infty$. This dispersive mechanism for the free transport (which can be seen as linearized Vlasov-Poisson around zero) is one of the key ingredients in the classical work \cite{Bardos-Degond} by Bardos and Degond. ~\\
Now we discuss the decay for one derivative of $\rho_{\text{free}}(t)$. Taking $\nabla_x$ on both sides of \eqref{free}, we get
\[
\nabla_x\rho_{\text{free}}(t)=t^{-d} t^{-1}\int_{\R^d}\nabla_v f_0\lw(w,\frac{x-w}{t}\rw)dw,
\]
and hence 
\[
\|\nabla_x\rho_{\text{free}}(t)\|_{L^\infty}\le t^{-(d+1)}\|\nabla_v f_0\|_{L^1_xL^\infty_v}\quad\text{and}\quad
\|\nabla_x\rho_{\text{free}}(t)\|_{L^1}\le t^{-1}\|\nabla_vf_0\|_{L^1_xL^1_v}.
\]
This implies that $\nabla_x\rho_{\text{free}}(t)$ decays like $t^{-d-1}$ in $L^\infty$ and $t^{-1}$ in $L^1$. This decaying mechanism is in fact extended to the Vlasov-Poisson system around vacuum by Hwang, Rendall and Velazquez \cite{Hwang}. In particular, the authors in \cite{Hwang} establish the improved decay estimates 
\beq\label{dev}
\|\nabla_x^k \rho(t)\|_{L^\infty}\lesssim (1+t)^{-d-k}\qquad \text{for any}\quad k\ge 0
\eeq 
for small initial data near vacuum. 

The natural question is whether the estimate \eqref{dev} still holds for solution to the screened Vlasov-Poisson under small perturbation around nonzero homogeneous equilibria $\mu(v)$ that satisfies Penrose condition \eqref{Penrose}. 
In this paper, we prove that this is essentially the case, namely $\nabla_x^k\rho(t)$ decays like $\nabla_x^k\rho_{\text{free}}(t)$, up to a log in time correction.
\subsection{Outline of the proof}
\textbf{The set up}:
Using the characteristics
\beq \label{char-eq}
\begin{cases} 
\frac{d}{ds}X_{s,t}(x,v)&=V_{s,t}(x,v),\quad X_{t,t}(x,v)=x,\\
\frac{d}{ds}V_{s,t}(x,v)&=E(s,X_{s,t}(x,v)),\quad V_{t,t}(x,v)=v,
\end{cases}\eeq
the solution $f(t,x,v)$ to the perturbation equation \eqref{Per-VP} can be written as 
\[
f(t,x,v)=f_0(X_{0,t}(x,v),V_{0,t}(x,v))-\int_0^t E(s,X_{s,t}(x,v))\cdot\nabla_v\mu(V_{s,t}(x,v))ds.
\]
Integrating both sides in $v\in\R^d$ and using the fact that $E=-\nabla_x(1-\triangle_x)^{-1}\rho$, we get
\beq\label{rho-eq}
\bega 
\rho(t,x)=\int_0^t \int_{\R^d}\nabla_x(1-\triangle_x)^{-1}\rho (s,x-(t-s)v)\cdot\nabla_v\mu (v)dvds+S(t,x)
\enda 
\eeq 
where $S(t,x)$ is given by
\beq\label{S-eq}
\bega
S(t,x)=&\int_{\R^d}f_0(X_{0,t}(x,v),V_{0,t}(x,v))dv\\
&+\int_0^t\int_{\R^d} E(s,x-(t-s)v)\cdot\nabla_v\mu(v)dvds-\int_0^t \int_{\R^d}E(s,X_{s,t}(x,v))\cdot\nabla_v\mu(V_{s,t}(x,v))dvds.
\enda 
\eeq 
Taking spacetime Fourier transform, one can express $\wtd \rho(\tau,\xi)$ as
\beq\label{spacetime}
\wtd \rho(\tau,\xi)=\frac{1}{1-\wtd K(\tau,\xi)}\wtd S(\tau,\xi),\qquad \text{where} \quad \wtd K(\tau,\xi)=\int_0^\infty e^{-i\tau t}\frac{i\xi}{1+|\xi|^2}\cdot\widehat{\nabla_v\mu}(t\xi)dt.
\eeq 
The Penrose stability condition \eqref{Penrose} is equivalent to 
\[
\inf_{\Im (\tau)\le 0}\inf_{\xi\in\R^d}|1-\wtd K(\tau,\xi)|\ge \kappa\quad\text{for some}\quad \kappa>0,
\] which is to avoid the singularity in \eqref{spacetime}. Following \cite{T-R-HK}, we can write in the original $(t,x)$ variables as follows:
\beq \label{rho-con}
\rho(t)=S(t)+\int_0^t G(t-s)\star_x S(s)ds
\eeq 
where 
\beq\label{G}
G(t,x)=\mathcal{F}^{-1}_{(\tau,\xi)\to (t,x)}\lw(\frac{\wtd{K}(\tau,\xi)}{1-\wtd{K}(\tau,\xi)}\rw).
\eeq
One of the main results in \cite{T-R-HK} was to derive pointwise estimates for the resolvent kernel $G$. ~\\
\textbf{Sketch of the proof of the main theorem}:~\\
Our first step is to derive higher derivative estimates for the Green kernel $G(t)$, by using Paley-Littlewood decomposition and localized frequency bounds of $G$. Making use of the decay bounds for $\nabla_x^k G(t)$, we are able to propagate the decay of $\nabla_x^k\rho(t)$ by the forcing term $S(t)$:
\beq\label{intro1}
\|\nabla_x^k \rho(t)\|_{L^1}\lesssim \la t\ra^{-k}\log(1+t)\|S\|_{Y_t^N}\qquad \text{and}\quad \|\nabla_x^k \rho(t)\|_{L^\infty}\lesssim \la t\ra^{-d-k}\log(1+t)\|S\|_{Y_t^N}.
\eeq 
where 

\[ 
\|S\|_{Y_t^N}=\max_{0\le k\le N}\sup_{0\le s\le t} \lw(\la s\ra^k\|\nabla_x^k S(s)\|_{L^1}+\la s\ra ^{d+k}\|\nabla_x^kS(s)\|_{L^\infty}\rw).
\]
Thus, it suffices to bound the derivatives of the forcing term $S(t,x)$, defined in \eqref{S-eq}. Inspired by \cite{T-R-HK, Bardos-Degond}, we show that the trajectories $(X_{s,t}(x,v),V_{s,t}(x,v))$ are closed to the characteristics of the free transport $(x-(t-s)v,v)$. To bound $\nabla_x^k S(t,x)$, we also need to use induction argument to decompose $\nabla_x^k S(t,x)$ into quantities involving $\nabla_x^k E,\nabla_x^k \rho$ and lower order derivative terms involving characteristic trajectories. 
\subsection{Organization of the paper}
The paper is organized as follows: In Section \ref{rho-decay}, we justify the estimate \eqref{intro1} for the density $\rho(t)$. In Section \ref{char-decay}, we prove the decay estimates for the characteristics. In Section \ref{sec4}, we bound the forcing term $S(t,x)$, by determining the forms of its derivatives (Lemma \ref{formI} and \ref{form}), and then justify the decay estimates for each of the terms (Proposition \ref{thm1} and Proposition \ref{deri}).
\subsection{Acknowledgement}
The author would like to thank Toan T. Nguyen for his many insightful discussions on the subject. The research was supported by the NSF under grant DMS-1764119. 
\subsection{Notations}
For any complex numbers $A,B$, we write $A\lesssim B$ or $A=O(B)$, to mean that there exists a universal constant $C_0>0$ such that $|A|\le C_0|B|$. ~\\
For a vector field $F(x)=(F_1(x),\cdots, F_q(x))\in\R^q$ with $x\in \R^m$, we denote $\nabla_x^k F$ to be the set of derivatives
\[
\{\pt^\al F_j:\quad 1\le j\le q,\quad |\al|= k\}.
\]
Moreover, for two vector functions $F,G$ defined on $x\in \R^m,y\in \R^n$, we denote $
(\nabla^u_x F)(\nabla^v_y G)
$
to be the set of all products $XY$, where $X\in \nabla_x^u F$ and $Y\in \nabla_y^v G$. ~\\For any two vectors $a=(a_i)_{1\le i\le p},b=(b_i)_{1\le i\le p}$, we denote $a\le b$ if $a_i\le b_i$ for all $1\le i\le d$.~\\
We also denote $\la t\ra =(1+t^2)^{1/2}$. It is obvious that $\la t\ra \lesssim 1+t\lesssim \la t\ra$ for all $t\ge 0$.~\\
For a function $f(x)$ with $x\in \R^d$, we denote $\hat f(\xi)$ to be the standard Fourier transform of $f$, given by the formula
\[
\hat f(\xi)=\int_{\R^d}f(x)e^{-ix\cdot\xi}dx.
\]
For a function $h(t,x)$ with $x\in \R^d$ and $t\ge 0$, the space-time Fourier transform of $h$ is denoted by $\wtd h(\tau,\xi)$, and is given by 
\[
\wtd h(\tau,\xi)=\int_0^\infty\int_{\R^d}h(t,x)e^{-i\tau t}e^{-ix\cdot\xi}dxdt.
\]
We also use the Paley-Littlewood decomposition on $\R^d$. Let $\chi\in [0,1]$ be a smooth compactly supported function on the annulus $\frac 1 4\le |\xi|\le 4$ and equal to one on the annulus $\frac 1 2\le |\xi|\le 2$. Define \[
\chi_q(\xi)=\chi\lw(\frac{\xi}{2^q}\rw)\qquad \text{for}\quad q\in\mathbb{Z}
\]
For $u\in \mathcal{S}'(\R^d)$, we have the Paley-Littlewood decomposition
\[
u=\sum_{q\in \mathbb{Z}}u_q,\qquad \text{where}\quad \hat u_q(\xi)=\hat u(\xi)\chi_q(\xi).
\]

\section{Linear estimates}\label{rho-decay}
\subsection{Dispersive estimates for the Green kernel}
Now let $G$ be the kernel defined as in \eqref{G}. Using Paley-Littlewood decomposition, we can decompose $G$ as  
\[
G=\sum_{q\in \mathbb{Z}}G_q.
\]
Now, we recall the following decaying bounds on the localized-in-frequency Green kernel $G_q$ in \cite{T-R-HK}.
\begin{lemma}\label{Toan} Let $\mu(v)$ be smooth, satisfy the bound \eqref{bound} and the Penrose stability condition \eqref{Penrose}. For any $K>0$, there exists $A\ge 1$ such that for every $\delta \in (0,1]$, and for every $q\in \mathbb{Z}$ with $2^q\ge A$:
\[
\|G_q(t)\|_{L^1}\lesssim  \frac {2^{q(1+\delta)}}{1+2^{2q}}\cdot \frac{1}{(1+2^q t)^K}\quad
\text{and}\quad \|G_q(t)\|_{L^\infty}\lesssim \frac{2^{q(d+1+\delta)}}{1+2^{2q}}\cdot\frac{1}{(1+2^q t)^K}.
\]
Moreover, for $2^q\le A$, one has 
\[
\|G_q(t)\|_{L^1}\lesssim  \frac{2^q}{(1+2^q t)^K}\quad \text{and}\quad
\|G_q(t)\|_{L^\infty}\lesssim \frac{2^{q(d+1)}}{(1+2^qt)^K}.
\]
\end{lemma}~\\
Making use of the decay bounds on the Green kernels $\{G_q\}_{q\in \mathbb{Z}}$, we establish the following bound on the derivatives of $G(t)$:
\begin{theorem} For any integer $k\ge 0$, there holds
\[
\|\nabla_x^k G(t)\|_{L^1}\lesssim t^{-k-1}\qquad \text{and}\quad \|\nabla_x^k G(t)\|_{L^\infty}\lesssim t^{-d-1-k}
\]
for $t>0$, where $G$ is the kernel defined in \eqref{G}.
\end{theorem}
\begin{proof} We take $K$ so that  $K>k+d+1$. We bound $\|\nabla_x^k G(t)\|_{L^1}$ as follows:
\[\bega 
\|\nabla_x^k G(t)\|_{L^1}&\lesssim \sum_{2^q\le A}2^{kq}\|G_q(t)\|_{L^1}+\sum_{2^q\ge A}2^{kq}\|G_q(t)\|_{L^1}\\
&\lesssim \sum_{2^q\le A} \frac{2^{q(k+1)}}{(1+2^q t)^K}+\sum_{2^q\ge A} \frac{2^{q(k+d+1+\delta)}}{(1+2^{2q})(1+2^q t)^K}.\\
\enda\]
For the first term, we see that
\[
\bega 
 \sum_{2^q\le A} \frac{2^{q(k+1)}}{(1+2^q t)^K}&\lesssim \lw(\sum_{2^q\le t^{-1}}+\sum_{t^{-1}\le 2^q\le 1}+\sum_{1\le 2^q\le A}\rw) \frac{2^{q(k+1)}}{(1+2^q t)^K}\\
 &\lesssim \sum_{2^q\le t^{-1}}2^{q(k+1)}+t^{-K}\sum_{2^q\le 1}2^{q(k+1-K)}+t^{-K}\sum_{1\le 2^q\le A}2^{q(k+1-K)} \\
 & \lesssim t^{-(k+1)}+t^{-K}\lesssim t^{-(k+1)}.
  \enda 
\]
The second term is treated as follows:
\[
\bega
\sum_{2^q\ge A} \frac{2^{q(k+d+1+\delta)}}{(1+2^{2q})(1+2^q t)^K}&\lesssim t^{-K}\sum_{2^q\ge A} 2^{q(k+d+1+\delta-2-K)}\lesssim t^{-K}\lesssim t^{-k-1}\qquad \text{since}\quad K>k+1+d.
\enda 
\]
The decay bound for $\|\nabla_x^k G(t)\|_{L^1}$ is complete. Now we bound $\|\nabla_x^k G(t)\|_{L^\infty}$ as follows:
\[\bega
\|\nabla_x^k G(t)\|_{L^\infty}&\lesssim \sum_{2^q\le A} 2^{kq}\|G_q(t)\|_{L^\infty}+\sum_{2^q\ge A}2^{kq}\|G_q(t)\|_{L^\infty}\\
&\lesssim \sum_{2^q\le A}\frac{2^{q(d+1+k)}}{(1+2^q t)^K}+\sum_{2^q\ge A}\frac{2^{q(d+1+\delta+k)}}{1+2^{2q}}\cdot\frac{1}{(1+2^qt)^K}.\\\enda\]
For the first term, we see that 
\[\bega 
\sum_{2^q\le A}\frac{2^{q(d+1+k)}}{(1+2^q t)^K}&\lesssim  \lw(\sum_{2^q\le t^{-1}}+\sum_{t^{-1}\le 2^q\le 1}+\sum_{1\le 2^q\le A}\rw)\frac{2^{q(d+1+k)}}{(1+2^qt)^K}\\
& \lesssim \sum_{2^q\le t^{-1}} 2^{q(d+1+k)}+t^{-K}\sum_{t^{-1}\le 2^q\le 1}2^{q(d+1+k-K)}+t^{-K}\sum_{1\le 2^q\le A} 2^{q(d+1+k-K)}\\
&\lesssim t^{-(d+1+k)}+t^{-K}\lesssim t^{-(d+1+k)}.
\enda 
\]
For the second term, we see that 
\[\bega 
\sum_{2^q\ge A}\frac{2^{q(d+1+\delta+k)}}{1+2^{2q}}\cdot\frac{1}{(1+2^qt)^K}&\lesssim t^{-K}\sum_{2^q\ge A}2^{q(d+1+\delta+k-K-2)}\lesssim t^{-K}\lesssim t^{-(d+1+k)}.
\enda 
\]
The proof is complete.
\end{proof}
\subsection{Decay estimates for the density and electric field of the linearized problem}
In this section, we drive decay estimates for higher derivatives of the density $\rho(t)$, defined in \eqref{rho-con} and and the electric field $E=-\nabla_x(1-\nabla_x)^{-1}\rho$. For $N>0$, we define
\beq 
\|S\|_{Y_t^N}=\max_{0\le k\le N}\sup_{0\le s\le t} \lw(\la s\ra^k\|\nabla_x^k S(s)\|_{L^1}+\la s\ra ^{d+k}\|\nabla_x^kS(s)\|_{L^\infty}\rw).
\eeq 
By definition, we see that 
\[
\|\nabla_x^k S(s)\|_{L^1}\le \la s\ra^{-k} \|S\|_{Y_t^N}\qquad \text{and}\quad \|\nabla_x^k S(s)\|_{L^\infty}\le \la s\ra^{-(k+d)}\|S\|_{Y_t^N}
\]
for $0\le k\le N$ and $0\le s\le t$.
\begin{theorem} \label{decay-rho} Let $N\ge 1$ be an integer. The density $\rho(t)$ defined in \eqref{rho-con} satisfies the bound
\beq\label{decayrho}
\max_{0\le k\le N}\lw\{\la t\ra ^k\|\nabla_x^k \rho(t)\|_{L^1}+\la t\ra^{d+k}\|\nabla_x^k \rho(t)\|_{L^\infty}\rw\}\lesssim \log(1+t)  \|S\|_{Y_t^N}\qquad \text{for}\quad t>0.
\eeq
\end{theorem}
\begin{proof}  Fix $k$ so that $0\le k\le N$. First we estimate $\nabla_x^k\rho$ in $L^1$. Applying $\nabla_x^k$ to both sides of \eqref{rho-con}, we get
\[
\nabla_x^k \rho(t)=\nabla_x^k S(t)+\int_0^{t/2} \nabla_x^k G(t-s)\star_x S(s)ds+\int_{t/2}^t G(t-s)\star_x \nabla_x^k S(s)ds.
\]
This implies 
\[\bega 
\|\nabla_x^k \rho(t)\|_{L^1}&\le \|\nabla_x^k S(t)\|_{L^1}+\int_0^{t/2}\|\nabla_x^k G(t-s)\|_{L^1}\|S(s)\|_{L^1}ds+\int_{t/2}^t \|G(t-s)\|_{L^1}\|\nabla_x^k S(s)\|_{L^1}ds\\
&\lesssim \la t\ra^{-k}\|S\|_{Y_t^N}+\|S\|_{Y_t^N}\lw(\int_0^{t/2}\frac{1}{(t-s)^{k+1}}ds+\int_{t/2}^t \frac{1}{(1+t-s)}\la s\ra^{-k}ds\rw)\\
&\lesssim \|S\|_{Y_t^N}\lw(\la t\ra^{-k}+t^{-k}+\frac{\log(1+t)}{(1+t)^k}\rw)
\enda 
\]
Hence 
\[
\la t\ra^k \|\nabla_x^k\rho(t)\|_{L^1}\lesssim \log(1+t)\|S\|_{Y_t^N}.
\]
Now we estimate the $L^\infty$ norm of $\nabla_x^k\rho$. We have 
\[\bega
\|\nabla_x^k\rho(t)\|_{L^\infty}&\le \|\nabla_x^k S(t)\|_{L^\infty}+\int_0^{t/2}\|\nabla_x^kG(t-s)\|_{L^\infty}\|S(s)\|_{L^1}ds+\int_{t/2}^t \|G(t-s)\|_{L^1}\|\nabla_x^k S(s)\|_{L^\infty}ds\\
&\lesssim \|S\|_{Y^N_t}\la t\ra^{-k-d}+\|S\|_{Y_t^N}\lw(\int_0^{t/2}\frac{1}{(t-s)^{k+d+1}}ds+\int_{t/2}^t \frac{1}{(1+t-s)}\la s\ra^{-k-d}ds\rw)\\
&\lesssim \frac{\log(1+t)}{(1+t)^{d+k}}\|S\|_{Y_t^N},
\enda 
\]
and hence 
\[
\la t\ra^{d+k}\|\nabla_x^k \rho(t)\|_{L^\infty}\lesssim \log(1+t)\|S\|_{Y_t^N}.
\]
The proof is complete.\end{proof}
\begin{theorem}
Assume that 
\[
\max_{0\le k\le N}\lw\{ \la t\ra^{d+k}\|\nabla_x^k\rho(t)\|_{L^\infty}+\la t\ra^k \|\nabla_x^k \rho(t)\|_{L^1}
\rw\}\lesssim \eps \log(1+t),
\]
there holds 
\[
\max_{0\le k\le N}\lw\{ \la t\ra^{d+k}\|\nabla_x^kE(t)\|_{L^\infty}+\la t\ra^k\|\nabla_x^kE(t)\|_{L^1}\rw\}\lesssim \eps\log(1+t).
\]
\end{theorem}
\begin{proof}
Since $E=-\nabla_x(1-\triangle_x)^{-1}\rho$, we have 
\[
\nabla_x^k E=-\nabla_x(1-\triangle_x)^{-1}(\nabla_x^k \rho)
\]
By the standard elliptic estimate, we get 
\[\begin{cases}
\|\nabla_x^k E(t)\|_{L^\infty}&\lesssim \|\nabla_x^k\rho(t)\|_{L^\infty}\lesssim \eps\la t\ra^{-d-k}\log(1+t),\\
\|\nabla_x^kE(t)\|_{L^1}&\lesssim \|\nabla_x^k\rho(t)\|_{L^1}\lesssim \eps\la t\ra^{-k}\log(1+t).
\end{cases}
\]
The proof is complete.
\end{proof}
\section{Decay estimates for the characteristics}\label{char-decay}
\subsection{Main theorem}
From the equations \eqref{char-eq}, the characteristics $X_{s,t}(x,v)$ and $V_{s,t}(x,v)$ can be written as follows:
\[
\begin{cases}
X_{s,t}(x,v)&=x-(t-s)v+\int_s^t (\tau-s)E(\tau,X_{\tau,t}(x,v))d\tau,\\
V_{s,t}(x,v)&=v-\int_s^t E(\tau,X_{\tau,t}(x,v))d\tau.
\end{cases}
\]
Following \cite{T-R-HK}, we define $(Y_{s,t}(x,v), W_{s,t}(x,v))$ so that 
\beq \label{rep-char}
\begin{cases}
X_{s,t}(x,v)=x-(t-s)v+Y_{s,t}(x-tv,v),\\
V_{s,t}(x,v)=v+W_{s,t}(x-tv,v).
\end{cases}
\eeq
Hence, we get 
\beq \label{Y-V}
\begin{cases}
Y_{s,t}(x,v)&=\int_s^t(\tau-s)E(\tau,x+\tau v+Y_{\tau,t}(x,v))d\tau,\\
W_{s,t}(x,v)&=-\int_s^t E(\tau,x+\tau v+Y_{\tau,t}(x,v))d\tau.
\end{cases}
\eeq 
Our main theorem in this section is as follows:
\begin{theorem}\label{EYW-decay} Assume that
\beq \label{E-decay}
\max_{0\le k\le N}\lw(\la t\ra^{d+k}\|\nabla_x^k E(t)\|_{L^\infty}\rw)\lesssim\eps \log(1+t)\qquad \text{for all}\quad t>0,
\eeq 
there holds, for all $s\in [0,t]$, the following inequalities:
\beq \label{Y-decay}
\max_{0\le k\le N}\|\nabla_v^k Y_{s,t}\|_{L^\infty}\lesssim \eps \frac{\log(1+s)}{(1+s)^{d-2}}
\eeq 
and 
\beq\label{W-decay}
\max_{0\le k\le N}\|\nabla_v^k W_{s,t}\|_{L^\infty}\lesssim \eps \frac{\log(1+s)}{(1+s)^{d-1}}.
\eeq  
\end{theorem} Before giving the proof, we recall the Faa di Bruno's formula in \cite{chemin}, which allows us to compute higher order derivatives of $Y_{s,t}$ and $W_{s,t}$ by the generalized chain rule:
\begin{lemma}\label{chain}
Let $u:\R^d\to \R^m$ and $F:\R^m\to \R$ be smooth functions. For each multi-index $\al\in \mathbb{N}^d$, we have 
\[
\pt^\al(F\circ u)=\sum_{\mu,\nu}C_{\mu,\nu}\pt^\mu F \prod_{1\le |\beta|\le |\al|,1\le j\le m}(\pt^\beta u^j)^{\nu_{\beta_j}}
\]
where $C_{\mu,\nu}$ are non negative integers, and the sum is taken over $\mu,\nu$ such that $1\le |\mu|\le |\al|$, $\nu_{\beta_j}\in \mathbb{N}^\star$, 
\[
\sum_{1\le |\beta|\le |\al|}\nu_{\beta_j}=\mu_j\qquad \text{for}\quad 1\le j\le m,\qquad\text{and}\quad \sum_{1\le |\beta|\le |\al|,1\le j\le m}\beta \nu_{\beta_j}=\al.
\]
\end{lemma}~\\
\textit{Proof of Theorem \ref{EYW-decay}}:
Fix $k\in\{0,1,\cdots,N\}$. We shall prove that 
\[
\la s\ra^{d-2}\|\nabla_v^k Y_{s,t}\|_{L^\infty}+\la s\ra^{d-1}\|\nabla_v^k W_{s,t}\|_{L^\infty}\lesssim \eps \log(1+s).
\]
In \cite{T-R-HK}, the authors prove the above statement when $k\in\{0,1\}$, thus we shall only consider $k\ge 2$.
Let us first justify the bound for the case $k=2$. Applying $\pt_{v_iv_j}^2$ to both sides of \eqref{Y-V}, we have 
\[\bega
\pt_{v_i v_j}^2Y_{s,t}(x,v)=&\int_s^t (\tau-s)\pt_{x_i x_j}^2E(\tau,x+\tau v+Y_{\tau,t}(x,v))\lw(\tau+\pt_{v_j}Y_{\tau,t}(x,v)\rw)\lw(\tau+\pt_{v_i}Y_{\tau,t}(x,v)\rw)d\tau\\
&+\int_s^t (\tau-s) \pt_{x_i}E(\tau,x+\tau v+Y_{\tau,t}(x,v))\pt_{v_i v_j}^2Y_{\tau,t}(x,v)d\tau.
\enda \]
This implies
\[
\bega 
|\pt_{v_i v_j}^2Y_{s,t}(x,v)|&\le \int_s^t (\tau-s)\|\nabla_x^2 E(\tau)\|_{L^\infty}(\tau+\|\nabla_v Y_{\tau,t}\|_{L^\infty})^2d\tau+\int_s^t (\tau-s)\|\nabla_x E(\tau)\|_{L^\infty}\sup_{0\le s\le t}\|\nabla_v^2Y_{s,t}\|_{L^\infty}d\tau\\
&\lesssim \int_s^t (\tau-s)\cdot \eps\log(1+\tau)\la \tau\ra^{-d-2}(\tau+\eps)^2d\tau+\sup_{0\le s\le t}\|\nabla_v^2 Y_{s,t}\|_{L^\infty}\int_s^t (\tau-s)\cdot\eps \log(1+\tau)\la \tau\ra^{-d-1}d\tau.
\enda 
\]
Hence we get 
\[
\|\nabla_v^2 Y_{s,t}\|_{L^\infty}\lesssim \eps \frac{\log(1+s)}{(1+s)^{d-2}}+\lw(\eps\frac{\log(1+s)}{(1+s)^{d-1}}\rw)\sup_{0\le s\le t}\|\nabla_v^2Y_{s,t}\|_{L^\infty}.
\]
Thus, as long as $\eps  \frac{\log(1+s)}{(1+s)^{d-1}}$ is small, we have 
\[
\|\nabla_v^2 Y_{s,t}\|_{L^\infty}\lesssim \eps \frac{\log(1+s)}{(1+s)^{d-2}}.
\]
Now we estimate $\nabla_v^2 W_{s,t}$. Applying $\pt_{v_iv_j}^2$ to both sides of\eqref{Y-V}, we get 
\[\bega
\pt^2_{v_i v_j}W_{s,t}(x,v)&=-\int_s^t \lw(\pt_{x_ix_j}^2E(\tau,x+\tau v+Y_{\tau,t}(x,v))(\tau+\pt_{v_j}Y_{\tau,t}(x,v))\rw)\lw(\tau+\pt_{v_i}Y_{\tau,t}(x,v)\rw)d\tau\\
&-\int_s^t \pt_{x_i}E(\tau,x+\tau v+Y_{\tau,t}(x,v))\lw(\pt^2_{v_i v_j}Y_{\tau,t}(x,v)\rw)d\tau.
\enda 
\]
This implies 
\[\bega
|\pt^2_{v_i v_j}W_{s,t}(x,v)|&\le \int_s^t \|\nabla_x^2 E(\tau)\|_{L^\infty}(\tau+\|\nabla_v Y_{\tau,t}\|_{L^\infty})^2+\int_s^t \|\nabla_x E(\tau)\|_{L^\infty}\|\nabla_v^2 Y_{\tau,t}\|_{L^\infty}d\tau\\
\enda
\]
Using the fact that $\la t\ra^d\|E(t)\|_{L^\infty}+\la t\ra^{d+1}\|\nabla_x E(t)\|_{L^\infty}\lesssim \eps\log(1+t)$  and $\|\nabla_v^2 Y_{\tau,t}\|_{L^\infty}\lesssim \eps$, we get
\[\bega
|\pt_{v_iv_j}^2W_{s,t}(x,v)|&\lesssim \eps\int_s^t \frac{\log(1+\tau)}{(1+\tau)^{d+2}} (\tau+\eps)^2d\tau+\eps^2\int_s^t \frac{\log(1+\tau)}{(1+\tau)^{d+1}}d\tau\\
&\lesssim \eps \int_s^t \frac{\log(1+\tau)}{(1+\tau)^d}d\tau+\eps^2 \frac{\log(1+s)}{(1+s)^d}\lesssim \eps \frac{\log(1+s)}{(1+s)^{d-1}}.
\enda \]
Now for a general multi-index $\al$ with $k=|\al|\ge 3$, we proceed by induction on $k$, by assuming that the decay estimates \eqref{Y-decay} and \eqref{W-decay} are true for all index with length less than $k$. 
Applying the chain rule \eqref{chain} for $Y_{s,t}$, we get 
\beq\label{DY}
\bega 
\pt_v^\al Y_{s,t}&=\int_s^t(\tau-s)\sum_{(\mu,\nu)\in I}C_{\mu,\nu}\pt_x^\mu E \prod_{1\le |\beta|\le |\al|,1\le j\le d}\lw(\pt_v^\beta(x_j+\tau v_j+Y^j_{\tau,t}(x,v))\rw)^{\nu_{\beta_j}}d\tau\\
&=\int_s^t(\tau-s)\sum_{(\mu,\nu)\in I}C_{\mu,\nu}\pt_x^\mu E \prod_{1\le |\beta|\le |\al|,1\le j\le d}\lw(\pt_v^\beta(\tau v_j+Y^j_{\tau,t}(x,v))\rw)^{\nu_{\beta_j}}d\tau\\
&=\int_s^t(\tau-s)\sum_{(\mu,\nu)\in I}C_{\mu,\nu}\pt_x^\mu E \prod_{1\le |\beta|\le |\al|,1\le j\le d}\tau^{\nu_{\beta_j}}\lw(\pt_v^\beta v_j\rw)^{\nu_{\beta_j}}d\tau\\
&\quad+\int_s^t(\tau-s)\sum_{(\mu,\nu)\in I}C_{\mu,\nu}\pt_x^\mu E \prod_{1\le |\beta|\le |\al|,1\le j\le d}\lw(\pt_v^\beta Y^j_{\tau,t}(x,v)\rw)^{\nu_{\beta_j}}d\tau.
\enda 
\eeq 
where 
\[
I=\lw\{(\mu,\nu)\in \mathbb{N}^{2d}:\quad \sum_{1\le |\beta|\le |\al|}\nu_{\beta_j}=\mu_j\qquad\text{for}\quad 1\le j\le d, \sum_{1\le |\beta|\le |\al|,1\le j\le d}\beta\nu_{\beta_j}=\al\rw\}.
\]
Now we fix $\mu,\nu$ and estimate each term appearing in  \eqref{DY}. The first term can be estimated as follows:
\[\bega 
\int_s^t (\tau-s)\pt_x^\mu E\prod_{1\le |\beta|\le |\al|,1\le j\le d}\tau^{\nu_{\beta_j}} (\pt_v^\beta v_j)^{\nu_{\beta_j}}d\tau&\lesssim \int_s^t (\tau-s)\frac{\eps\log(1+\tau)}{(1+\tau)^{d+|\mu|}}\tau^{\sum_{j=1}^d\nu_{\beta_j}}d\tau\\
&\lesssim \int_s^t \tau^{\mu_j+1}\frac{\eps\log(1+\tau)}{(1+\tau)^{d+|\mu|}}d\tau\\
&\le \int_s^t \eps\frac{\log(1+\tau)}{(1+\tau)^{d-1}}d\tau\lesssim \eps \frac{\log(1+s)}{(1+s)^{d-2}}.\\
\enda 
\]
Now we estimate the second term in \eqref{DY}, which is
\beq\label{2nd}
\int_s^t(\tau-s)\pt_x^\mu E \prod_{1\le |\beta|\le |\al|,1\le j\le d}\lw(\pt_v^\beta Y^j_{\tau,t}(x,v)\rw)^{\nu_{\beta_j}}d\tau.
\eeq 
We consider two cases:~\\
\textbf{Case 1:} $\beta_j<|\al|$ for all $1\le j\le d.$~\\
In this case, we use the induction hypothesis on $\pt_v^\beta Y_{\tau,t}$, which is 
$
|\pt_v^\beta Y_{\tau,t}^j(x,v)|\lesssim \eps 
$.
Hence \eqref{2nd} can be bounded by 
\[\bega 
&\int_s^t (\tau-s)\|\pt_x^\mu E(\tau)\|_{L^\infty}\eps^{\sum_{j=1}^d\nu_{\beta_j}} d\tau\lesssim \int_s^t \tau \frac{\eps \log(1+\tau)}{(1+\tau)^{d+|\mu|}}\eps^{|\mu|}d\tau\lesssim \eps \frac{\log(1+s)}{(1+s)^{d-2}}.
\enda 
\]
\textbf{Case 2:} There exists $j_0\in \{1,2,\cdots, d\}$ such that $\beta_{j_0}=|\al|$. ~\\
In this case, since $|\beta|\le |\al|$, we have $\beta_j=0$ for all $j\neq j_0$. Hence the term is reduced to 
\[
\int_s^t (\tau-s)\pt_x^\mu E\cdot (\pt_{v_{j_0}}^{|\al|}Y_{\tau,t}^j)^{\nu_{|\al|}}O(\eps)d\tau,
\]
where we use the fact that $\|Y_{s,t}\|_{L^\infty}\lesssim \eps$. Moreover, since $\sum_{j=1}^d \beta \nu_{\beta_j}=\al$, we have $\nu_{|\al|}\le 1$. Hence 
\[\bega
&\int_s^t (\tau-s)\pt_x^\mu E\cdot (\pt_{v_{j_0}}^{|\al|}Y_{\tau,t}^j)^{\nu_{|\al|}}O(\eps)\lesssim \eps \int_s^t \tau \frac{\eps\log(1+\tau)}{(1+\tau)^{d+|\mu|}}\sup_{0\le s\le t}\|\pt_v^{|\al|}Y_{s,t}\|_{L^\infty}d\tau\\
&=\sup_{0\le s\le t}\|\pt_v^{|\al|}Y_{s,t}\|_{L^\infty}\int_s^t \eps^2 \frac{\log(1+\tau)}{(1+\tau)^{d+|\mu|-1}}d\tau\lesssim \eps^2 \sup_{0\le s\le t}\|\nabla_v^{|\al|}Y_{s,t}\|_{L^\infty}.
\enda 
\]
Thus, we get the inequality of the form 
\[
\|\nabla_v^\al Y_{s,t}\|_{L^\infty}\lesssim \eps \frac{\log(1+s)}{(1+s)^{d-2}}+\eps^2 \sup_{0\le s\le t}\|\nabla_v^{|\al|}Y_{s,t}\|_{L^\infty}.
\]
Hence for $\eps$ small, we get 
\[
\|\nabla_v^\al Y_{s,t}\|_{L^\infty}\lesssim \eps \frac{\log(1+s)}{(1+s)^{d-2}}\qquad\text{for}\quad \text{all}\quad 0\le s\le t.
\]
The proof is complete.
\subsection{Straightening the characteristics}

Next, we recall the following lemma about straightening the characteristics from \cite{T-R-HK}.
\begin{lemma} \label{D-Yst}
Let $Y_{s,t}$ be defined as in \eqref{rep-char} and \eqref{Y-V}. Assume that $E$ satisfies the decay estimates \eqref{E-decay}, there holds 
\[
\|\nabla_x Y_{s,t}\|_{L^\infty}\lesssim \frac{\eps\log(1+s)}{(1+s)^{d-1}}.
\]
\end{lemma}
\begin{proof}
By the definition of $Y_{s,t}$, we have 
\[
Y_{s,t}(x,v)=\int_s^t (\tau-s)E(\tau,x+\tau v+Y_{\tau,t}(x,v))d\tau
\]
Hence we get 
\[\bega
|\nabla_x Y_{s,t}(x,v)|&\le \int_s^t (\tau-s)\|\nabla_x E(\tau)\|_{L^\infty}(1+\|\nabla_x Y_{\tau,t}\|_{L^\infty})d\tau\\
&\lesssim \int_s^t (\tau-s)\frac{\eps\log(1+\tau)}{(1+\tau)^{d+1}}d\tau+\sup_{0\le \tau \le t}\|\nabla_x Y_{\tau,t}\|_{L^\infty}\int_s^t (\tau-s)\frac{\eps\log(1+\tau)}{(1+\tau)^{d+1}}d\tau\\
&\lesssim \eps \frac{\log(1+s)}{(1+s)^{d-1}}+\lw\{\eps\frac{\log(1+s)}{(1+s)^{d-1}} \rw\}\sup_{0\le s\le t}\|Y_{s,t}\|_{L^\infty}.
\enda 
\]
Hence as long as $\eps$ is small enough, we have 
\[
|\nabla_x Y_{s,t}(x,v)|\lesssim \frac{\eps\log(1+s)}{(1+s)^{d-1}}
\]
The proof is complete.
\end{proof}
\begin{lemma}\label{straighten}
For $0\le s\le t$, there exists a $C^1$ map $(x,v)\to \Psi_{s,t}(x,v)$ such that 
\[
X_{s,t}(x,\Psi_{s,t}(x,v))=x-(t-s)v
\]
for all $x,v\in \R^d$. Moreover, if $E$ satisfies the estimates \eqref{E-decay}, there holds
\[
\la s\ra^d|\Psi_{s,t}(x,v)-v|+\la s\ra^{d-1}|\nabla_v(\Psi_{s,t}(x,v)-v)|\lesssim \eps\log(1+s)
\]
for all $x,v\in \R^d$ and $0\le s\le t$.
\end{lemma}
\begin{proof}
We write 
\[
X_{s,t}(x,v)=x-(t-s)(v+\Phi_{s,t}(x,v))
\]
and will show that $(x,v)\to (x,v+\Phi_{s,t}(x,v))$ is a $C^1$ differomorphism. To this end, we prove that 
\[
\la s\ra^d \|\Phi_{s,t}\|_{L^\infty}+\la s\ra^d \|\nabla_x\Phi_{s,t}\|_{L^\infty}+\la s\ra^{d-1}\|\nabla_v\Phi_{s,t}\|_{L^\infty}\lesssim \eps\log(1+s)
\]
We have 
\beq\label{Phi}
\bega
\Phi_{s,t}(x,v)&=-\frac{1}{t-s}(X_{s,t}(x,v)-(t-s)v)\\
&=-\frac{1}{t-s}\int_s^t (\tau-s)E(\tau,x-(t-\tau)v+Y_{\tau,t}(x-vt,v))d\tau\\
&\lesssim \frac{1}{t-s}\int_s^t (\tau-s)\|E(\tau)\|_{L^\infty}d\tau\lesssim \int_s^t \|E(\tau)\|_{L^\infty}d\tau.
\enda 
\eeq
Since $E=-\nabla_x(1-\triangle_x)^{-1}\rho$, we have 
\[
\|E(\tau)\|_{L^\infty}\lesssim \|\nabla_x \rho(\tau)\|_{L^\infty}\lesssim \frac{\eps\log(1+\tau)}{(1+\tau)^{d+1}}.
\]
Thus we have 
\[
|\Phi_{s,t}(x,v)|\lesssim \int_s^t \frac{\eps\log(1+\tau)}{(1+\tau)^{d+1}}d\tau\lesssim \eps \frac{\log(1+s)}{(1+s)^d}.
\]
Thus $\la s\ra^d\|\Phi_{s,t}\|_{L^\infty}\lesssim \eps\log(1+s)$.
Next, applying $\nabla_x$ to both sides of \eqref{Phi}, we have 
\[
\|\nabla_x\Phi_{s,t}\|_{L^\infty}\lesssim \frac{1}{t-s}\int_s^t (\tau-s)\|\nabla_x E(\tau)\|_{L^\infty}\lw(1+\|\nabla_xY_{\tau,t}\|_{L^\infty}\rw)d\tau.
\]
Now using lemma \ref{D-Yst}, we get
\[
\|\nabla_x \Phi_{s,t}\|_{L^\infty}\lesssim \int_s^t\frac{\eps\log(1+\tau)}{(1+\tau)^{d+1}}d\tau\lesssim \eps\la s\ra^{-d}\log(1+s).
\]
Now we bound $\|\nabla_v\Phi_{s,t}\|_{L^\infty}$. Applying $\nabla_v$ to both sides of \eqref{Phi}, we get 
\[
\|\nabla_v\Phi_{s,t}\|_{L^\infty}\lesssim \frac{1}{t-s}\int_s^t (\tau-s)\|\nabla_x E(\tau)\|_{L^\infty}\lw\{(t-\tau)+t\|\nabla_xY_{\tau,t}\|_{L^\infty}
\rw\}d\tau.
\]
Again, using Lemma \ref{D-Yst}, we have
\[\bega
\|\nabla_v\Phi_{s,t}\|_{L^\infty}&\lesssim \frac{1}{t-s}\int_s^t (\tau-s)\frac{\eps\log(1+\tau)}{(1+\tau)^{d+1}}\lw((t-\tau)+\frac{\eps t\log(1+\tau)}{(1+\tau)^{d-1}}\rw)d\tau\\
&\lesssim \eps \frac{\log(1+s)}{(1+s)^{d-1}}+\eps\int_s^t\frac{ t(\tau-s)}{t-s}\frac{\log(1+\tau)^2}{(1+\tau)^{d-1}}d\tau\\
&\lesssim \eps\frac{\log(1+s)}{(1+s)^{d-1}}+\eps\int_s^t (\tau-s)\frac{\log(1+\tau)^2}{(1+\tau)^{d-1}}d\tau+\eps\int_s^t (\tau-s)^2\frac{\log(1+\tau)^2}{(1+\tau)^{d-1}}d\tau\\
&\lesssim \eps\frac{\log(1+s)}{(1+s)^{d-1}}.
\enda 
\]
Hence, the map $(x,v)\to(x,v+\Phi_{s,t}(x,v))$ is a $C^1$ differomorphism. Thus there exists a $C^1$ differomorphism $v\to \Psi_{s,t}(x,v)$ such that 
\[
X_{s,t}(x,\Psi_{s,t}(x,v))=x-(t-s)v.
\] 
Combining this with $X_{s,t}(x,v)=x-(t-s)(v+\Phi_{s,t}(x,v))$, we have 
\[\begin{cases}
|\Psi_{s,t}(x,v)-v|&\lesssim \|\Phi_{s,t}\|_{L^\infty}\lesssim \eps\la s\ra^{-d}\log(1+s),\\
|\nabla_v(\Psi_{s,t}(x,v)-v)|&\lesssim \|\nabla_v\Phi_{s,t}\|_{L^\infty}\lesssim \eps \la s\ra^{-(d-1)}\log(1+s).
\end{cases}
\]
The proof is complete.
\end{proof}
\section{Decay estimates for the forcing term}\label{sec4}
In this section, we derive decay estimates for the derivatives of the forcing term, appearing in the equation \eqref{rho-eq}. The forcing term $S(t,x)$ is given by 
\beq\label{S}
\bega
S(t,x)=&\int_{\R^d}f_0(X_{0,t}(x,v),V_{0,t}(x,v))dv+\int_0^t\int_{\R^d} E(s,x-(t-s)v)\cdot\nabla_v\mu(v)dvds\\
&-\int_0^t \int_{\R^d}E(s,X_{s,t}(x,v))\cdot\nabla_v\mu(V_{s,t}(x,v))dvds\\
&=\mathcal{I}(t,x)+\mathcal{R}_L(t,x)-\mathcal{R}_{NL}(t,x),
\enda 
\eeq
where
\[\begin{cases}
\mathcal I(t,x)&=\int_{\R^d}f_0(X_{0,t}(x,v),V_{0,t}(x,v))dv,\\
\mathcal R_L(t,x)&=\int_0^t\int_{\R^d}E(s,x-(t-s)v)\cdot\nabla_v\mu(v)dvds,\\
\mathcal R_{NL}(t,x)&=\int_0^t\int_{\R^d}E(s,X_{s,t}(x,v))\cdot\nabla_v\mu(V_{s,t}(x,v))dvds.
\end{cases}
\]
Fix $N\ge 1$, we will give decay estimates for $\|S\|_{Y_t^N}$, under the decaying assumptions on derivatives of $E$, $Y_{s,t}$ and $W_{s,t}$.
We also recall the smallness assumption for the initial perturbation $f_0(x,v)$ as follows:
\beq\label{small}
\max_{0\le k\le N}\|\nabla_{x,v}^{k}f_0\|_{L^1_xL^\infty_v}\le \eps_0.
\eeq
Our main theorem is as follows:
\begin{theorem}
Let $N>1$ be an integer. Assume that $E, W_{s,t},Y_{s,t}$ satisfy the decay estimates \eqref{E-decay}, \eqref{W-decay} and \eqref{Y-decay} and $f_0$ satisfies the smallness assumption \eqref{small}, there holds 
\beq\label{I-bound}
\sup_{0\le k\le N}\lw(\la t\ra^k\|\nabla_x^k \mathcal I (t)\|_{L^1}+\la t\ra^{d+k}
\|\nabla_x^k \mathcal I (t)\|_{L^\infty}\rw)\lesssim\eps_0,
\eeq 
and 
\[
\sup_{0\le k\le N} \lw(
\la t\ra^k\|\nabla_x^k\mathcal R (t)\|_{L^1}+\la t\ra^{d+k}\|\nabla_x^k \mathcal R (t)\|_{L^\infty}
\rw)\lesssim \eps^2.
\]

\end{theorem}
\subsection{Decay estimates for the initial data term}
First, we estimate $\mathcal I(t,x)$, under suitable smallness assumption \eqref{small} on the initial data $f_0$. In \cite{T-R-HK}, the authors prove that 
\[
\|\mathcal I (t)\|_{L^1}+\la t\ra^d\|\mathcal I(t)\|_{L^\infty}+\la t\ra \|\nabla_x\mathcal I(t)\|_{L^1}+\la t\ra^{d+1}\|\nabla_x \mathcal I(t)\|_{L^\infty}\lesssim \eps_0.
\]
Hence, we shall establish the bound \eqref{I-bound} for $k\ge 2$.
First, we establish the following lemma 
\begin{lemma}\label{formI} Let $k\ge 2$ and
\[
\mathcal I (t,x)=\int_{\R^d}f_0(X_{0,t}(x,v),V_{0,t}(x,v))dv.
\]
The term $\nabla_x^k \mathcal{I}(t,x)$ can be written as a sum of many terms, which are all in the  form
\beq\label{term}
\int_{\R^d}\nabla_{x,v}^{\al}f_0\cdot (\nabla_v^{\beta_1} Y_{0,t})^{k_1}\cdots (\nabla_v^{\beta_r}Y_{0,t})^{k_r}\cdot (\nabla_v^{\gamma_1} W_{0,t})^{s_1}\cdots(\nabla_v^{\gamma_t}W_{0,t})^{s_t}\dfrac{dw}{t^{d+k}}
\eeq 
where $\al$, $(\beta_1,k_1),\cdots,(\beta_r,k_2)$, and $(\gamma_1,s_1),\cdots,(\gamma_t,s_t)$ satisfy
\beq\label{index}
1\le |\al|\le k\qquad \text{and}\qquad  (\beta_1k_1+\cdots+\beta_rk_r)+(\gamma_1s_1+\cdots+\gamma_t s_t)\le k.
\eeq 
\end{lemma}
\begin{proof}
From \eqref{rep-char}, we get
\[
\mathcal{I}(t,x)=\int_{\R^d}f_0(x-tv+Y_{0,t}(x-tv,v),v+W_{0,t}(x-tv,v))dv.
\]
Let $w=x-tv$, we get 
\[
\mathcal I (t,x)=\int_{\R^d}f_0\lw(w+Y_{0,t}(w,\frac{x-w}{t}),\frac{x-w}{t}+W_{0,t}(w,\frac{x-w}{t})\rw)\frac{dw}{t^d}.
\]
By a direct calculation, we have 
\beq\label{diff2-I}
\bega 
\pt^2_{x_i x_j}\mathcal{I}(t,x)&=\int_{\R^3}\lw(\pt_{x_i x_j}f_0\pt_{v_j}Y_{0,t}+\pt_{x_i v_j}f_0(1+\pt_{v_j}W_{0,t})\rw)\lw(\pt_{v_i}Y_{0,t}\rw)\frac{dw}{t^{d+2}}+\int_{\R^3}\lw(\pt_{x_i} f_0\rw)\lw(\pt_{v_i v_j} Y_{0,t}\rw)\frac{dw}{t^{d+2}}\\
&+\int_{\R^3}\lw(\pt_{x_j v_i}f_0\pt_{v_j}Y_{0,t}+\pt_{v_i v_j}f_0(1+\pt_{v_j}W_{0,t})\rw)\lw(\pt_{v_i}W_{0,t}\rw)\frac{dw}{t^{d+2}}+\int_{\R^3}\lw(\pt_{v_i}f_0\rw)\lw(\pt_{v_i v_j}W_{0,t}\rw)\frac{dw}{t^{d+2}}.
\enda
\eeq 
It is clear from the above that  \eqref{diff2-I} that 
\[\bega
\nabla_x^2 \mathcal{I}(t,x)=&\int_{\R^d}\{(\nabla_x^2 f_0)(\nabla_v Y_{0,t})+(\nabla_{x,v}f_0)(\nabla_v Y_{0,t})+(\nabla_x f_0)(\nabla_v^2 Y_{0,t})+(\nabla_{x,v}^2f_0)(\nabla_v Y_{0,t})\\
&+(\nabla_v^2 f_0)(\nabla_v W_{0,t})+(\nabla_v^2f_0)(\nabla_v W_{0,t})^2+(\nabla_v f_0)(\nabla_v^2 W_{0,t})\}\frac{dw}{t^{d+2}}
\enda 
\]
which satisfies the hypothesis for $k=2$. Now by induction, we assume that this statement is true for $k$, and we shall prove it for $k+1$. Applying $\pt_{x_i}$ to the term \eqref{term} and using the product rules, we have three types of terms that appear, namely: 
\[\bega
I_1&=\int_{\R^d}\dfrac{d}{dx_i}(\nabla_{x,v}^\al f_0)(\nabla_v^{\beta_1} Y_{0,t})^{k_1}\cdots (\nabla_v^{\beta_r}Y_{0,t})^{k_r}\cdot (\nabla_v^{\gamma_1} W_{0,t})^{s_1}\cdots(\nabla_v^{\gamma_t}W_{0,t})^{s_t}\dfrac{dw}{t^{d+k}}\\
I_2&=\int_{\R^d}(\nabla_{x,v}^\al f_0)(\nabla_v Y_{0,t})(\nabla_v^{\beta_1} Y_{0,t})^{k_1-1}\cdots (\nabla_v^{\beta_r}Y_{0,t})^{k_r}\cdot (\nabla_v^{\gamma_1} W_{0,t})^{s_1}\cdots(\nabla_v^{\gamma_t}W_{0,t})^{s_t}\dfrac{dw}{t^{d+k+1}}\\
I_3&=\int_{\R^d}(\nabla_{x,v}^\al f_0)(\nabla_v^{\beta_1} Y_{0,t})^{k_1}\cdots (\nabla_v^{\beta_r}Y_{0,t})^{k_r}\cdot
(\nabla_vW_{0,t}) (\nabla_v^{\gamma_1} W_{0,t})^{s_1-1}\cdots(\nabla_v^{\gamma_t}W_{0,t})^{s_t}\dfrac{dw}{t^{d+k+1}}
\enda
\]
Here, $I_1, I_2,I_3$ appear when $\pt_{x_i}$ hits $\nabla_{x,v}^\al f_0, (\nabla_v^{\beta_1}Y_{0,t})^{k_1}$ and $(\nabla_v^{\gamma_1}W_{0,t})^{s_1}$ respectively. Note that we assume that the derivative hits the above terms on just the pairs $(\beta_1,k_1)$ or $(\gamma_1,s_1)$, as this is up to a permutation of indices.~\\
\textbf{Treating $I_1$:} ~\\ By a direct calculation, we see that $I_1$ can be written as 
\[
\int_{\R^d}\lw(\nabla_{x,v}^{\al+1}f_0\cdot\nabla_v Y_{0,t}+\nabla_{x,v}^{\al+1}f_0(1+\nabla_v W_{0,t})
\rw)(\nabla_v^{\beta_1} Y_{0,t})^{k_1}\cdots (\nabla_v^{\beta_r}Y_{0,t})^{k_r}\cdot (\nabla_v^{\gamma_1} W_{0,t})^{s_1}\cdots(\nabla_v^{\gamma_t}W_{0,t})^{s_t}\dfrac{dw}{t^{d+k+1}}\\
\]
which satisfies the induction hypothesis for $|\al|+1=k+1$.~\\
\textbf{Treating $I_2$:} ~\\ For $I_2$, we check the condition \eqref{index} for the new indices and multi-indices, which is 
\[
|\al|\le k+1\qquad \text{and}\quad 1+\beta_1(k_1-1)+(\beta_2k_2+\cdots+\beta_rk_r)+(\gamma_1s_1+\cdots+\gamma_ts_t)\le k+1
\]
The statement $|\al|\le k+1$ is true, as $|\al|\le k$. For the second condition, we note that, by the induction hypothesis:
\[
\lw(\beta_1k_1+\beta_2k_2+\cdots+\beta_rk_r\rw)+\lw(\gamma_1s_1+\cdots+\gamma_ts_t\rw)\le k.
\]
Adding both sides of the above by $1-\beta_1$, the new left hand side can be bounded by $k+1-\beta_1\le k+1$, since $\beta_1\ge 0$. The proof is complete.
Finally, the term $I_3$ is treated exactly as $I_2$, and we skip the details. 
\end{proof}
\begin{theorem}\label{thm1}Assume that 
\[
\max_{0\le k\le N}\|\nabla_{x,v}^{k}f_0\|_{L^1_xL^\infty_v}\le \ep_0\]
and $E, Y_{s,t},W_{s,t}$ satisfy the decay estimates \eqref{E-decay}, \eqref{Y-decay}, and \eqref{W-decay} for $0\le k\le N$. Then 
\[
\mathcal{I}(t,x)=\int_{\R^d}f_0(X_{0,t}(x,v),V_{0,t}(x,v))dv
\]
satisfies the following decay estimate:
\[
\sup_{0\le k\le N}\lw(\la t\ra^k\|\nabla_x^k \mathcal I (t)\|_{L^1}+\la t\ra^{d+k}
\|\nabla_x^k \mathcal I (t)\|_{L^\infty}\rw)\lesssim\eps_0.
\]
\end{theorem}
\begin{proof}
By the above lemma, it suffices to prove that 
\[
\la t\ra^k\|\nabla_x^k \mathcal J (t)\|_{L^1}+\la t\ra^{d+k}
\|\nabla_x^k \mathcal J (t)\|_{L^\infty}\lesssim\eps_0.
\]
where 
\[\begin{cases}
&\mathcal J=\int_{\R^d}\nabla_{x,v}^{\al}f_0\cdot (\nabla_v^{\beta_1} Y_{0,t})^{k_1}\cdots (\nabla_v^{\beta_r}Y_{0,t})^{k_r}\cdot (\nabla_v^{\gamma_1} W_{0,t})^{s_1}\cdots(\nabla_v^{\gamma_t}W_{0,t})^{s_t}\dfrac{dw}{t^{d+k}}\\
&1\le |\al|\le k\qquad \text{and}\qquad  (\beta_1k_1+\cdots+\beta_rk_r)+(\gamma_1s_1+\cdots+\gamma_t s_t)\le k.
\end{cases}
\]
We have 
\[
\mathcal J (t,x)\lesssim t^{-k}\int_{\R^d}|\nabla_{x,v}^\al f_0|(X_{0,t}(x,v),V_{0,t}(x,v))dv
\]
Hence, we get 
\[
t^k \|\mathcal J(t)\|_{L^1}+t^{d+k}\|\mathcal J(t)\|_{L^\infty}\lesssim \eps_0.
\]
The proof is complete.
\end{proof}
\subsection{Decay estimates for the reaction term}
In this section, we estimate the derivatives of the reaction term 
\[
\mathcal{R}=\mathcal{R}_{L}-\mathcal{R}_{NL}=
\int_0^t \int_{\R^d} E(s,x-(t-s)v)\cdot\nabla_v\mu(v)dvds
-\int_0^t \int_{\R^d}E(s,X_{s,t}(x,v))\cdot\nabla_v\mu(V_{s,t}(x,v))dvds\]
appearing as a forcing term in \eqref{S}.  For a general time-dependent vector field $E(s)$ and a smooth decaying function $\mu$, we also denote $\mathcal T$ to be 
\beq\label{bi}
\mathcal{T}(E,\mu)=\int_0^t \int_{\R^d} E(s,x-(t-s)v)\cdot\nabla_v\mu(v)dvds
-\int_0^t \int_{\R^d}E(s,X_{s,t}(x,v))\cdot\nabla_v\mu(V_{s,t}(x,v))dvds
\eeq 
Our main theorem is as follows:
\begin{theorem} Let $N>1$ be an integer. Assume that $E, W_{s,t},Y_{s,t}$ satisfy the decay estimates \eqref{E-decay}, \eqref{W-decay} and \eqref{Y-decay} for all $0\le k\le N$, there holds 
\[
\max_{0\le k\le N} \lw(
\la t\ra^k\|\nabla_x^k\mathcal R (t)\|_{L^1}+\la t\ra^{d+k}\|\nabla_x^k \mathcal R (t)\|_{L^\infty}
\rw)\lesssim \eps^2.
\]
\end{theorem}~\\
First, we recall the following proposition from \cite{T-R-HK}. We also give a detailed proof for the readers convenience.
\begin{proposition}\label{assume}
Assuming that $E,Y_{s,t},W_{s,t}$ satisfies the decaying estimates \eqref{E-decay},\eqref{Y-decay} and \eqref{W-decay} respectively, we have 
\[
\|\mathcal T (E,\mu)\|_{L^1}+\la t\ra^d\|\mathcal T (E,\mu)\|_{L^\infty}\lesssim \eps^2.
\]
\end{proposition}
\begin{proof}
Making the change of variables $v\to \Psi_{s,t}(x,v)$ so that $X_{s,t}(x,\Psi_{s,t}(x,v))=x-(t-s)v$ (see Lemma \ref{straighten}),
we have 
\[\bega
\mathcal T(E,\mu)&=\int_0^t \int_{\R^d}E(s,x-(t-s)v)\cdot\nabla_v\mu(v)dvds-\int_0^t \int_{\R^d}E(s,X_{s,t}(x,v))\cdot\nabla_v\mu(V_{s,t}(x,v))dvds\\
&=\int_0^t \int_{\R^d}E(s,x-(t-s)v)\cdot\nabla_v\mu(v)dvds\\
&\quad-\int_0^t\int_{\R^d}E(s,x-(t-s)v)\cdot\nabla_v\mu(V_{s,t}(x,\Psi_{s,t}(x,v))\det(\nabla_v \Psi_{s,t}(x,v))dvds
\enda
\]
Hence, one can rewrite $\mathcal T$ as $\mathcal T_1+\mathcal T_2$, where 
\[\bega 
\mathcal T_1&=\int_0^t\int_{\R^d}E(s,x-(t-s)v)\cdot\nabla_v \lw\{\mu(v)-\mu(V_{s,t}(x,\Psi_{s,t}(x,v)))\rw\}dvds\\
\mathcal T_2&=\int_0^t\int_{\R^d}E(s,x-(t-s)v)\cdot\nabla_v\mu(V_{s,t}(x,\Psi_{s,t}(x,v))\lw\{1-\det(\nabla_v\Psi_{s,t}(x,v))
\rw\}dvds
\enda 
\]
\textbf{Bounding} $\mathcal T_1(t)$:~\\
We have 
\[\bega
\mathcal T_1(t,x)&=\int_0^t\int_{\R^d}E(s,x-(t-s)v)\cdot\nabla_v\lw\{\mu(v)-\mu(V_{s,t}(x,\Psi_{s,t}(x,v)))\rw\}dvds\\
&\lesssim \int_0^t \int_{\R^d}|E(s,x-(t-s)v)\cdot|\la v\ra^{-M}\cdot|v-V_{s,t}(x,\Psi_{s,t}(x,v))|dvds\\
&\lesssim \int_0^t \int_{\R^d}|E(s,x-(t-s)v)|\la v\ra^{-M}\lw\{|v-V_{s,t}(x,v)|+|V_{s,t}(x,v)-V_{s,t}(x,\Psi_{s,t}(x,v))|\rw\}dvds\\
&\lesssim \int_0^t \int_{\R^d}|E(s,x-(t-s)v)|\la v\ra^{-M}\lw\{|v-V_{s,t}(x,v)|_{L^\infty}+\|\nabla_v V_{s,t}\|_{L^\infty}|v-\Psi_{s,t}(x,v)|\rw\}dvds
\enda \]
Using the fact that 
\beq \label{bound-char}
|v-V_{s,t}(x,v)|\lesssim \frac{\eps\log(1+s)}{(1+s)^{d-1}}, \quad |v-\Psi_{s,t}(x,v)|\lesssim \frac{\eps\log(1+s)}{(1+s)^d}\quad \text{and}\quad \|\nabla_v V_{s,t}\|_{L^\infty}\lesssim 1,
\eeq 
we get 
\[
\mathcal T_1(t,x)\lesssim \int_0^t \int_{\R^d}|E(s,x-(t-s)v)|\la v\ra^{-M}\lw(\frac{\eps\log(1+s)}{(1+s)^{d-1}}+\frac{\eps\log(1+s)}{(1+s)^{d}}
\rw)dvds.
\]
Hence 
\[\bega 
\|\mathcal T_1(t)\|_{L^1}&\lesssim \int_0^t\frac{\eps\log(1+s)}{(1+s)^{d-1}}\int_{\R^d}\|E(s)\|_{L^1}\la v\ra^{-M}dv\\
&\lesssim \int_0^t \frac{\eps^2\log(1+s)^2}{(1+s)^{d-1}}ds\lesssim \eps^2,
\enda 
\]
and 
\[\bega 
\|\mathcal T_1(t)\|_{L^\infty}&\lesssim \int_0^t\int_{\R^d} \|E(s)\|_{L^\infty}\la v\ra^{-M}\lw(\frac{\eps\log(1+s)}{(1+s)^{d-1}}+\frac{\eps\log(1+s)}{(1+s)^{d}}
\rw)dvds\\
&\lesssim \int_0^t\frac{\eps^2\log(1+s)^2}{(1+s)^{2d-1}}ds\lesssim \eps^2\log(1+t)\int_0^t \frac{\log(1+s)}{(1+s)^{2d-1}}ds\\
&\lesssim \eps^2 \frac{\log(1+t)^2}{(1+t)^{2d-2}}\lesssim \eps^2\la t\ra^{-d}.
\enda
\]
Thus $$\|\mathcal T_1(t)\|_{L^1}+\la t\ra^d\|\mathcal T_1(t)\|_{L^\infty}\lesssim \eps^2.$$~\\
\textbf{Bounding} $\mathcal T_2(t)$:~\\
Since $v\to \Psi_{s,t}(x,v)$ is a differomorphism, making the change of variables $\Psi_{s,t}^{-1}:\Psi_{s,t}(x,v)\to v$ gives 
\[\bega
\mathcal T_2(t,x)&=\int_0^t\int_{\R^d}E(s,x-(t-s)v)\cdot\nabla_v\mu(V_{s,t}(x,v))(1-\det(\nabla_v\Psi_{s,t}(x,v))\det(\nabla_v(\Psi_{s,t}^{-1}(x,v)))dvds\\
&\lesssim \int_0^t \int_{\R^d}|E(s,x-(t-s)v)|\cdot |\nabla_v\mu|\lw(V_{s,t}(x,v)\rw)\cdot\lw|\det(\nabla_v\Psi_{s,t}(x,v))^{-1}-1\rw|dvds
\enda
\]
Using the inequality $|\nabla_v(\Psi_{s,t}(x,v)-v)|\lesssim \frac{\eps\log(1+s)}{(1+s)^{d-1}}$, we have 
\[
\mathcal T_2(t,x)\lesssim \int_0^t\int_{\R^d}|E(s,x-(t-s)v)|\cdot|\nabla_v\mu|(V_{s,t}(x,v))\cdot\frac{\eps\log(1+s)}{(1+s)^{d-1}}dsdv
\]
Hence we have 
\[\bega
\|\mathcal T_2(t)\|_{L^1}&\lesssim \int_0^t\lw\{ \eps\frac{\log(1+s)}{(1+s)^{d-1}}\|E(s)\|_{L^1}\int_{\R^d}\sup_{x\in \R^d}|\nabla_v\mu|(V_{s,t}(x,v))dv\rw\}ds\\
&\lesssim \int_0^t \eps^2\frac{\log(1+s)^2}{(1+s)^d}ds\lesssim \eps^2.
\enda
\]
and 
\[\bega 
\|\mathcal T_2(t)\|_{L^\infty}&\lesssim \int_0^t \|E(s)\|_{L^\infty}\cdot\frac{\eps\log(1+s)}{(1+s)^{d-1}}\sup_{x\in\R^d}\int_{\R^d}|\nabla_v\mu|(V_{s,t}(x,v))dvds\\
&\lesssim \int_0^t \eps^2\frac{\log(1+s)^2}{(1+s)^{2d-1}}ds\lesssim \eps^2\log(1+t)\int_0^t\frac{\log(1+s)}{(1+s)^{2d-1}}ds\lesssim \eps^2\frac{\log(1+t)}{(1+t)^{2d-2}}\\
&\lesssim \eps^2\la t\ra^{-d}
\enda
\]
This implies that 
\[
\|\mathcal T_2(t)\|_{L^1}+\la t\ra^d\|\mathcal T_2(t)\|_{L^\infty}\lesssim \eps^2.
\]
The proof is complete.
\end{proof}
\begin{lemma}\label{chain-T}
There holds 
\[\bega
\pt_j \mathcal{T}(E,\mu)&=\frac{1}{t}\lw(\mathcal{T}(s\pt_j E,\mu)+\mathcal{T}(E,\pt_j \mu)\rw)\\
&+\frac{1}{t}\sum_{k=1}^d \int_0^t \int_{\R^d}\pt_{v_j}Y_{s,t}(x-tv,v)\cdot\nabla_x E_k(s,X_{s,t}(x,v))(\pt_k\mu)(V_{s,t}(x,v))dvds\\
&+\frac{1}{t}\sum_{k=1}^d\int_0^t\int_{\R^d}E_k(s,X_{s,t}(x,v))(\pt_{v_j}W_{s,t})(x-tv,v)\cdot\nabla_v (\pt_k\mu)(V_{s,t}(x,v))dvds. \\
\enda 
\]
\end{lemma}
\begin{proof} 
We recall that
\[
\mathcal{R}_{NL}=\int_0^t \int_{\R^d}E(s,X_{s,t}(x,v))\cdot\nabla_v \mu(V_{s,t}(x,v))dvds.
\]
Using the identities \eqref{rep-char}, we have
\[
\mathcal R_{NL}=\int_0^t \int_{\R^d}E(s,x-(t-s)v+Y_{s,t}(x-tv,v))\cdot\nabla_v \mu(v+W_{s,t}(x-tv,v))dvds.
\]
Making the change of variable $w=x-tv$, we obtain
\[
\bega 
\mathcal R_{NL}&=\int_0^t\int_{\R^d}E\lw(s,w+\frac{s}{t}(x-w)+Y_{s,t}(w,\frac{x-w}{t})\rw)\cdot
\nabla_v \mu\lw(\frac{x-w}{t}+W_{s,t}(w,\frac{x-w}{t})\rw)
t^{-d}dwds\\
&=t^{-d}\sum_{k=1}^d\int_0^t \int_{\R^d}E_k\lw(s,w+\frac{s}{t}(x-w)+Y_{s,t}(w,\frac{x-w}{t})\rw)\pt_k \mu\lw(\frac{x-w}{t}+W_{s,t}(w,\frac{x-w}{t})\rw)
dwds\enda
\]
Similarly, for $\mathcal R_L$ we have 
\[
\bega 
\mathcal{R}_L=\int_0^t \int_{\R^d}E(s,x-(t-s)v)\cdot\nabla_v\mu(v)dvds
=t^{-d}\sum_{k=1}^d\int_0^t \int_{\R^d}E_k\lw(s,w+\frac{s}{t}(x-w)\rw)\pt_k\mu\lw(\frac{x-w}{t}\rw)dwds\\
\enda 
\]
The lemma follows by a direct calculation. The proof is complete.
\end{proof}
Now we establish the following lemma, by induction on the degree of derivatives.
\begin{lemma}\label{form} Let $n\ge 2$, then $\nabla_x^n \mathcal R(t,x)$ can be written as a sum of terms, are all either of the form 
\[
\frac{1}{t^n}\mathcal T(s^k\nabla_x^k E, f(\mu))
\]
or the form 
\[\bega
\frac{1}{t^n}\int_0^t\int_{\R^d}&\lw\{(\nabla_v^{m_1}Y_{s,t})^{k_1}\cdots (\nabla_v^{m_a}Y_{s,t})^{k_a}\rw\} \cdot \lw\{(\nabla_v^{n_1}W_{s,t})^{l_1}\cdots (\nabla_v^{n_b}W_{s,t})^{l_b}\rw\}\\
&\times \lw\{(s^{t_1}\nabla_x^{u_1}E)\cdots (s^{t_c}\nabla_x^{u_c}E)
\rw\}f(\mu)dvds
\enda 
\]
where $f(\mu)$ is some expression only depending on $\mu(v)$ or its derivatives. Moreover, the indices satisfy the following set of conditions: 
\begin{itemize}
\item{No loss of derivative condition: $k\le n$,  $\max\{\{m_1,k_1,\cdots, m_a,k_a\}\cup\{n_1,l_1,\cdots, n_b,l_b\}\cup\{t_1,u_1,\cdots, t_c,u_c\}\}\le n$.}
\item{E-decay condition: $(t_1,t_2,\cdots, t_c)\le (u_1,u_2,\cdots, u_c).$}
\item{Y-W show-up condition: $\min \{a,b\} \ge 1$.}
\item{E-show up condition: $c\ge 1$.}
\item{$W$-$E$ decay condition: If  $b=0$ then  $t_h+1\le u_h$ for some $1\le h \le c$.}
\end{itemize}
\end{lemma}
Let us call the first form to be type-I and the second one is type-II.
\begin{remark}The conditions on the indices are important for the decay estimates. The first condition means that we do not lose derivatives in the estimates. The second condition requires a good decay for the quantities $s^{t_i}\nabla_x^{u_i}E$. The third condition means that at least $Y_{s,t}$ or $W_{s,t}$ (or their derivatives) shows up in the expression. The forth condition means that $E$ or its derivatives must show up in the expression. Finally, the last condition means that if $W_{s,t}$ (or its derivatives)  does not show up, then we have more gradient of $E$ to control the power of $s$. The reason for the last condition will be clear when we estimate \eqref{clear} later in the paper.
\end{remark} 
\begin{proof}
The lemma is proved by induction on $n$. Assuming the lemma is true for $n$, we justify the above claim for $n+1$.~\\
\textbf{Gradient of type-I terms:}~\\
Applying $\pt_j$ to the term $\mathcal{T}(s^k \nabla_x^k E,f(\mu))$ and using Lemma \ref{chain-T}, we have 
\[\bega
&\frac{1}{t^n}\pt_j T(s^k\nabla_x^kE,f(\mu))=\frac{1}{t^{n+1}}\lw(T(s^{k+1}\nabla_x^k \pt_j E,f(\mu))+T(s^k \nabla_x^k E,\pt_j f(\mu))\rw)\\
&+\frac{1}{t^{n+1}}\int_0^t \int_{\R^d}\lw(\lw(\pt_{v_j}Y_{s,t}(x-tv,v)\rw)\cdot \nabla_x(s^k\nabla_x^kE)\lw(s,X_{s,t}(x,v)\rw)(\nabla_vf(\mu))\lw(V_{s,t}(x,v))\rw)\rw)dvds\\
&+\frac{1}{t^{n+1}}\int_0^t \int_{\R^d}(s^k\nabla_x^k E)\lw(s,X_{s,t}(x,v))\rw)\lw\{\pt_{v_j}W_{s,t}(x-tv,v)\cdot\nabla_v^2 f(\mu)(V_{s,t}(x,v))\rw\}dvds.\enda 
\]
We can see that all of the above terms are either type-I or type-II, and satisfy the induction hypothesis with order $n+1$.~\\
\textbf{Gradient of type-II form:}~\\
Now we consider the type-II term:
\[\bega 
\mathcal D=\frac{1}{t^n}\int_0^t\int_{\R^d}&\lw\{(\nabla_v^{m_1}Y_{s,t})^{k_1}\cdots (\nabla_v^{m_a}Y_{s,t})^{k_a}\rw\} \cdot \lw\{(\nabla_v^{n_1}W_{s,t})^{l_1}\cdots (\nabla_v^{n_b}W_{s,t})^{l_b}\rw\}\\
&\times \lw\{(s^{t_1}\nabla_x^{u_1}E)\cdots (s^{t_c}\nabla_x^{u_c}E)
\rw\}f(\mu)dvds.
\enda 
\]
Here $(\nabla_v^{m_i}Y_{s,t})^{k_i}, (\nabla_v^{n_j}W_{s,t})^{l_j}$ is evaluated at $(x-tv,v)$ and $\nabla_x^{u_r}E$ is evaluated at $(s,X_{s,t}(x,v))$.
Making the change of variables $w=x-tv$, we can rewrite $\mathcal D$ as
\[\bega
t^{-d}\frac{1}{t^n}\int_0^t\int_{\R^d}&\lw\{(\nabla_v^{m_1}Y_{s,t})^{k_1}\cdots (\nabla_v^{m_a}Y_{s,t})^{k_a}\rw\} 
\lw\{(\nabla_v^{n_1}W_{s,t})^{l_1}\cdots (\nabla_v^{n_b}W_{s,t})^{l_b}\rw\}\\
&\times \lw\{(s^{t_1}\nabla_x^{u_1}E)\cdots (s^{t_c}\nabla_x^{u_c}E)
\rw\}f(\mu)dwds.
\enda 
\]
where 
$(\nabla_v^{m_i}Y_{s,t})^{k_i}, (\nabla_v^{n_j}W_{s,t})^{l_j}$ is evaluated at $(w,\frac{x-w}{t})$ and $\nabla_x^{u_r}E$ is evaluated at 
\[\lw
(s,Y_{s,t}(w,\frac{x-w}{t})+w+\frac{s}{t}(x-w)\rw)
\]
Applying $\nabla_x$ to $\mathcal D$ and using the product rules, we have \[\begin{cases}
\mathcal D_1&=t^{-d}\frac{1}{t^n}\int_0^t \int_{\R^d}k_1\lw(\lw(\nabla_v^{m_1}Y_{s,t}
\rw)^{k_1-1}\lw\{\nabla_v^{m_1+1}Y_{s,t}
\rw\}\frac{1}{t}\rw)\cdots (\nabla_v^{m_a}Y_{s,t})^{k_a}\\
&\quad\lw\{(\nabla_v^{n_1}W_{s,t})^{l_1}\cdots (\nabla_v^{n_b}W_{s,t})^{l_b}\rw\}\lw\{(s^{t_1}\nabla_x^{u_1}E)\cdots (s^{t_c}\nabla_x^{u_c}E)
\rw\}f(\mu)dwds,\\
\mathcal D_2&=t^{-d}\frac{1}{t^n}\int_0^t \int_{\R^d}
\lw\{(\nabla_v^{m_1}Y_{s,t})^{k_1}\cdots (\nabla_v^{m_a}Y_{s,t})^{k_a}\rw\} \\
&\quad\lw(l_1\lw\{(\nabla_v^{n_1}W_{s,t})^{l_1-1}(\nabla_v^{n_1+1}W_{s,t})\rw\}\frac{1}{t}\rw)\cdots (\nabla_v^{n_b}W_{s,t})^{l_b}\\
&\quad \lw\{(s^{t_1}\nabla_x^{u_1}E)\cdots (s^{t_c}\nabla_x^{u_c}E)
\rw\}f(\mu)dwds,\\
\mathcal D_3=&t^{-d}\frac{1}{t^n}\int_0^t \int_{\R^d}\lw\{(\nabla_v^{m_1}Y_{s,t})^{k_1}\cdots (\nabla_v^{m_a}Y_{s,t})^{k_a}\rw\} 
\lw\{(\nabla_v^{n_1}W_{s,t})^{l_1}\cdots (\nabla_v^{n_b}W_{s,t})^{l_b}\rw\}\\
&\lw(\frac{1}{t}\lw\{s^{t_1}(\nabla_x^{u_1+1}E)(\nabla_v Y_{s,t})\rw\}+\frac{1}{t}\lw\{s^{t_1+1}\nabla_x^{u_1+1} E\rw\}\rw)\cdots(s^{t_c}\nabla_x^{u_c}E)f(\mu)dwds.
\end{cases}
\]
Now making the change of variables $v=\frac{x-w}{t}$, we get 
\[\begin{cases}
\mathcal D_1&=\frac{1}{t^{n+1}}\int_0^t \int_{\R^d}k_1\lw(\lw(\nabla_v^{m_1}Y_{s,t}
\rw)^{k_1-1}\lw\{\nabla_v^{m_1+1}Y_{s,t}
\rw\}\rw)\cdots (\nabla_v^{m_a}Y_{s,t})^{k_a}\\
&\quad\lw\{(\nabla_v^{n_1}W_{s,t})^{l_1}\cdots (\nabla_v^{n_b}W_{s,t})^{l_b}\rw\}\lw\{(s^{t_1}\nabla_x^{u_1}E)\cdots (s^{t_c}\nabla_x^{u_c}E)
\rw\}f(\mu)dvds,\\
\mathcal D_2&=\frac{1}{t^{n+1}}\int_0^t \int_{\R^d}
\lw\{(\nabla_v^{m_1}Y_{s,t})^{k_1}\cdots (\nabla_v^{m_a}Y_{s,t})^{k_a}\rw\} \\
&\quad\lw(l_1\lw\{(\nabla_v^{n_1}W_{s,t})^{l_1-1}(\nabla_v^{n_1+1}W_{s,t})\rw\}\rw)\cdots (\nabla_v^{n_b}W_{s,t})^{l_b}\\
&\quad \lw\{(s^{t_1}\nabla_x^{u_1}E)\cdots (s^{t_c}\nabla_x^{u_c}E)
\rw\}f(\mu)dvds,\\
\mathcal D_3=&\frac{1}{t^{n+1}}\int_0^t \int_{\R^d}\lw\{(\nabla_v^{m_1}Y_{s,t})^{k_1}\cdots (\nabla_v^{m_a}Y_{s,t})^{k_a}\rw\} 
\lw\{(\nabla_v^{n_1}W_{s,t})^{l_1}\cdots (\nabla_v^{n_b}W_{s,t})^{l_b}\rw\}\\
&\lw\{s^{t_1}(\nabla_x^{u_1+1}E)(\nabla_v Y_{s,t})+s^{t_1+1}\nabla_x^{u_1+1} E\rw\}\cdots(s^{t_c}\nabla_x^{u_c}E)f(\mu)dvds.
\end{cases}
\]
From the above expressions, it is straightforward that the new terms are of type-II, which satisfy the induction hypothesis with $n+1$. The proof is complete.
\end{proof}
\begin{proposition}\label{deri}
Let $\mathcal R$ be defined as in \eqref{bi}. There holds, for $n\ge 2$, the decaying estimates 
\[
\la t\ra^n\|\nabla_x^n \mathcal R (t)\|_{L^1}+\la t\ra^{d+n}\|\nabla_x^k\mathcal R (t)\|_{L^\infty}\lesssim \eps^2.
\]

\end{proposition}
\begin{proof}
By Lemma \ref{form}, $\nabla_x^n\mathcal R (t,x)$ can be decomposed as a sum of many terms, all are either of the form $\mathcal R_1$ or $\mathcal R_2$, where
\[\begin{cases}
\mathcal R _1&=\frac{1}{t^n}\mathcal{T}(s^k\nabla_x^kE,f(\mu))\\
\mathcal R _2&=\frac{1}{t^n}\int_0^t\int_{\R^d}\lw\{(\nabla_v^{m_1}Y_{s,t})^{k_1}\cdots (\nabla_v^{m_a}Y_{s,t})^{k_a}\rw\} \cdot \lw\{(\nabla_v^{n_1}W_{s,t})^{l_1}\cdots (\nabla_v^{n_b}W_{s,t})^{l_b}\rw\}\\
&\lw\{(s^{t_1}\nabla_x^{u_1}E)\cdots (s^{t_c}\nabla_x^{u_c}E)
\rw\}f(\mu)dvds
\end{cases}
\]
where 
\[
k\le n\qquad \text{and}\qquad \max\{\{m_1,k_1,\cdots, m_a,k_a\}\cup\{n_1,l_1,\cdots, n_b,l_b\}\cup\{t_1,u_1,\cdots, t_c,u_c\}\}\le n.
\]
\textbf{Bounding} $\mathcal R_1(t)$:~\\
We have 
\[
\mathcal R_1(t)=\frac{1}{t^n}\mathcal T(s^k \nabla_x^k E,f(\mu))
\]
with $k\le n$ and $f(\mu)$ is a decay function in $v$ which only depends on $\mu$ or its derivatives.
Applying Proposition \ref{assume}, we only need to check the assumption
\[
\|s^k\nabla_x^k E(s)\|_{L^\infty}\lesssim  \frac{\eps \log(1+s)}{(1+s)^{d}},
\]
which is true, since $\|\nabla_x^k E(s)\|_{L^\infty}\lesssim \frac{\eps\log(1+s)}{(1+s)^{d+k}}$. The proof for $\mathcal R_1$ is complete.~\\
\textbf{Bounding} $\mathcal R_2(t)$:~\\
Now we bound $\mathcal R_2(t)$, where 
\[\bega
\mathcal R _2(t,x)&=\frac{1}{t^n}\int_0^t\int_{\R^d}\lw\{(\nabla_v^{m_1}Y_{s,t})^{k_1}\cdots (\nabla_v^{m_a}Y_{s,t})^{k_a}\rw\} \cdot \lw\{(\nabla_v^{n_1}W_{s,t})^{l_1}\cdots (\nabla_v^{n_b}W_{s,t})^{l_b}\rw\}\\
&\lw\{(s^{t_1}\nabla_x^{u_1}E)\cdots (s^{t_c}\nabla_x^{u_c}E)
\rw\}f(\mu)dvds
\enda\]
Here $(\nabla_v^{m_i}Y_{s,t})^{k_i}, (\nabla_v^{n_j}W_{s,t})^{l_j}$ is evaluated at $(x-tv,v)$ and $\nabla_x^{u_r}E$ is evaluated at $(s,X_{s,t}(x,v))$.~\\
We also recall the following set of conditions, proved in Lemma \ref{form}:
\beq\label{set-of-con}
\begin{cases}
&k\le n\quad\text{and}\quad  \max\{\{m_1,k_1,\cdots, m_a,k_a\}\cup\{n_1,l_1,\cdots, n_b,l_b\}\cup\{t_1,u_1,\cdots, t_c,u_c\}\}\le n,\\
&(t_1,t_2,\cdots, t_c)\le (u_1,u_2,\cdots, u_c),\\
&\min \{a,b\} \ge 1,\\
& c\ge 1\\
&\text{If}\quad b=0\quad\text{then}\quad t_h+1\le u_h\quad\text{for some}\quad 1\le h \le c.
\end{cases}
\eeq
First, we will show that $\|\mathcal R_2(t)\|_{L^1}\lesssim \eps^2\la t\ra^{-n}$. ~\\
Using  the third condition in \eqref{set-of-con} and the fact that 
\[
\max_{0\le k\le n}\lw(\|\nabla_v^k Y_{s,t}\|_{L^\infty}+\|\nabla_v^k W_{s,t}\|_{L^\infty}\rw)\lesssim \eps \frac{\log(1+s)}{(1+s)^{d-2}},
\]
we get
\beq\label{R_2-now}
\mathcal R_2(t,x)\lesssim t^{-n}\int_0^t \int_{\R^d} \eps \frac{\log(1+s)}{(1+s)^{d-2}} \prod_{i=1}^c \lw(s^{t_i}\|\nabla_x^{u_i}E(s)\|_{L^\infty}\rw)|f(\mu)|(V_{s,t}(x,v))dvds.
\eeq 
Now using the second and the forth condition in \eqref{set-of-con}, and the decay of $\nabla_x^{u_i}E$, we get 
 \[
 s^{t_i}\|\nabla_x^{u_i}E(s)\|_{L^\infty}\lesssim \frac{\eps\log(1+s)}{(1+s)^{d+u_i-t_i}}\lesssim \frac{\eps\log(1+s)}{(1+s)^d}.
 \]
Applying the above inequality to \eqref{R_2-now}, we obtain
\[\bega 
\mathcal R_2(t,x)&\lesssim t^{-n}\int_0^t \int_{\R^d}\eps^2 \cdot \frac{\log(1+s)}{(1+s)^{d-2}}\cdot\frac{\log(1+s)}{(1+s)^d}|f(\mu)|(V_{s,t}(x,v))dvds\\
&\lesssim \eps^2 t^{-n}\int_0^t \int_{\R^d}\frac{\log(1+s)^2}{(1+s)^{2d-2}}|f(\mu)|(V_{s,t}(x,v))dvds.
\enda 
\]
Hence, using the decaying assumption of $\mu$, we obtain
\[
\|\mathcal R_2(t)\|_{L^1}\lesssim \eps^2 t^{-n}\int_0^t \frac{\log(1+s)^2}{(1+s)^{2d-2}}\lw(\int_{\R^d\times \R^d}|f(\mu)|(V_{s,t}(x,v))dxdv\rw)ds\lesssim \eps^2 t^{-n}.
\]
The decaying bound for $\|\mathcal R_2(t)\|_{L^1}$ is complete. \\
We split the integral in $\mathcal R_2$ into $\int_0^{t/2}+\int_{t/2}^t$, so that $\mathcal R_2=\mathcal R_3+\mathcal R_4$, where 
\[
\begin{cases}
\mathcal R_3(t,x)&=\frac{1}{t^n}\int_{t/2}^t\int_{\R^d}\lw\{(\nabla_v^{m_1}Y_{s,t})^{k_1}\cdots (\nabla_v^{m_a}Y_{s,t})^{k_a}\rw\} \cdot \lw\{(\nabla_v^{n_1}W_{s,t})^{l_1}\cdots (\nabla_v^{n_b}W_{s,t})^{l_b}\rw\}\\
&\lw\{(s^{t_1}\nabla_x^{u_1}E)\cdots (s^{t_c}\nabla_x^{u_c}E)
\rw\}f(\mu)dvds\\
\mathcal R_4(t,x)&=\frac{1}{t^n}\int_0^{t/2}\int_{\R^d}\lw\{(\nabla_v^{m_1}Y_{s,t})^{k_1}\cdots (\nabla_v^{m_a}Y_{s,t})^{k_a}\rw\} \cdot \lw\{(\nabla_v^{n_1}W_{s,t})^{l_1}\cdots (\nabla_v^{n_b}W_{s,t})^{l_b}\rw\}\\
&\lw\{(s^{t_1}\nabla_x^{u_1}E)\cdots (s^{t_c}\nabla_x^{u_c}E)
\rw\}f(\mu)dvds
\end{cases}
\]
For $\mathcal R_3(t,x)$, by the same argument for $\mathcal R_1(t,x)$, we have the pointwise bound 
\[\bega 
\mathcal R_3(t,x)&\lesssim \eps^2 t^{-n}\int_{t/2}^t\frac{\log(1+s)^2}{(1+s)^{2d-2}}\int_{\R^d}|f(\mu)|(V_{s,t}(x,v))dvds\\
&\lesssim \eps^2 t^{-n-d}\int_{t/2}^t \frac{\log(1+s)^2}{(1+s)^{d-2}}ds\lesssim \eps^2 t^{-(n+d)}
\enda 
\]
Thus $\|\mathcal R_3(t)\|_{L^\infty}\lesssim\eps^2 t^{-(n+d)}$.\\
Now to bound $\|\mathcal R_4(t)\|_{L^\infty}$, we use the inequalities 
\[
\max_{0\le k\le n}\lw(\la s\ra^{d-2}\|\nabla_v^k Y_{s,t}\|_{L^\infty}+\la s\ra^{d-1}\|\nabla_v^k W_{s,t}\|_{L^\infty}\rw)\lesssim \eps\log(1+s)
\]
to get
\[\bega
\mathcal R_4(t,x)&\lesssim t^{-n}\int_0^{t/2} \int_{\R^d} \lw(\frac{\eps\log(1+s)}{(1+s)^{d-2}}\rw)^a\lw(\frac{\eps\log(1+s)}{(1+s)^{d-1}}
\rw)^b\\
&\quad \times \prod_{i=1}^c \lw(s^{t_i}|\nabla_x^{u_i}E|(s,X_{s,t}(x,v))\rw)\cdot|f(\mu)|(V_{s,t}(x,v))dvds\\
 &\lesssim t^{-n}\eps^{a+b}\int_0^{t/2}\int_{\R^d}\frac{(\log(1+s))^{a+b}}{(1+s)^{a(d-2)+b(d-1)}} \\
 &\quad \times \prod_{i=1}^c\lw\{s^{t_i}|\nabla_x^{u_i}E|(s,X_{s,t}(x,v))\rw\}\cdot|f(\mu)|(V_{s,t}(x,v))dvds
 \enda 
\]
Using the change of variable $v\to \Psi_{s,t}(x,v)$ so that $X_{s,t}(x,\Psi_{s,t}(x,v))=x-(t-s)v$, we have 
\[\bega
\mathcal R_4(t,x)&\lesssim   t^{-n}\eps^{a+b}\int_0^{t/2}\int_{\R^d}\frac{(\log(1+s))^{a+b}}{(1+s)^{a(d-2)+b(d-1)}} \prod_{i=1}^c\lw\{s^{t_i}|\nabla_x^{u_i}E|(s,x-(t-s)v)\rw\}\\
&\cdot  |f(\mu)|(V_{s,t}(x,\Psi_{s,t}(x,v))) |\det(\nabla_v \Psi_{s,t}(x,v))|dvds
\enda 
\]
Now making the change of variables $w=x-(t-s)v$, we have 
\[\bega
\mathcal R_4(t,x)&\lesssim t^{-n}\eps^{a+b}\int_0^{t/2}(t-s)^{-d}\int_{\R^d} \frac{(\log(1+s))^{a+b}}{(1+s)^{a(d-2)+b(d-1)}}\prod_{i=1}^c\lw\{s^{t_i}|\nabla_x^{u_i}E|(s,w)\rw\}\\
&\cdot|f(\mu)|\lw(V_{s,t}(x,\Psi_{s,t}(x,\frac{x-w}{t})\rw)dwds\\
&\lesssim t^{-(n+d)}\eps^{a+b}\int_0^{t/2}\frac{\log(1+s)^{a+b}}{(1+s)^{a(d-2)+b(d-1)}}\cdot \min_{1\le i\le c} \lw\{s^{t_i}\|\nabla_x^{u_i}E(s)\|_{L^1}
\rw\}ds\\
&\lesssim t^{-(n+d)}\eps^{a+b}\int_0^{t/2}\frac{\log(1+s)^{a+b}}{(1+s)^{a(d-2)+b(d-1)}}\cdot\min_{1\le i\le c}\lw\{s^{t_i}\frac{\eps\log(1+s)}{(1+s)^{u_i}}\rw\}ds\\
&\lesssim t^{-(n+d)}\eps^{a+b+1}\int_0^{t/2}\frac{(\log(1+s))^{a+b+1}}{(1+s)^{\{a(d-2)+b(d-1)+\max_{1\le i\le c}(u_i-t_i)\}}}ds.
\enda
\]
Now using the third condition in \eqref{set-of-con}, we get $\eps^{a+b+1}\lesssim \eps^2$ as long as $\eps$ is small.
Hence we get 
\beq\label{clear}
\|\mathcal R_4(t)\|_{L^\infty}\lesssim \eps^2 t^{-(n+d)}\int_0^{t/2}\frac{(\log(1+s))^{a+b+1}}{(1+s)^{\{a(d-2)+b(d-1)+\max_{1\le i\le c}(u_i-t_i)\}}}ds.
\eeq
Thus, it suffices to prove that 
\[
\mathcal C=\int_0^{t/2}\frac{(\log(1+s))^{a+b+1}}{(1+s)^{\{a(d-2)+b(d-1)+\max_{1\le i\le c}(u_i-t_i)\}}}ds\lesssim 1.
\]
We consider two cases:~\\
\textit{Case 1}: $b=0$.~\\
In this case, we use use the last condition listed in \eqref{set-of-con} to get
$\max_{1\le i\le c}(u_i-t_i)\ge 1$, and hence
\[
\mathcal C\lesssim \int_0^{t/2}\frac{(\log(1+s))^{a+1}}{(1+s)^{a(d-2)+1}}\lesssim \int_0^{t/2}\frac{(\log(1+s))^{a+1}}{(1+s)^{a(d-2)+1}}ds\lesssim 1.
\]
\textit{Case 2}: $b\ge 1$.~\\
In this case, we estimate $\mathcal C$ as follows:
\[
\mathcal C\lesssim \int_0^{t/2}\frac{(\log(1+s))^a}{(1+s)^{a(d-2)}}\cdot\frac{(\log(1+s))^{b+1}}{(1+s)^{b(d-1)}}ds\lesssim \int_0^{t/2}\frac{(\log(1+s))^{b+1}}{(1+s)^{b(d-1)}}ds\lesssim 1.
\]
The proof is complete.
\end{proof}

Finally, we give the proof for the main Theorem \ref{main-thm}.~\\
\textit{Proof of Theorem} \ref{main-thm}.
Let 
\[
\mathcal N(t)=\sup_{0\le s\le t}\max_{0\le k\le N}\frac{\lw(
\la s\ra^k\|\nabla_x^k\rho(s)\|_{L^1}+\la s\ra^{k+d}\|\nabla_x^k \rho(s)\|_{L^\infty}
\rw)}{\log(1+s)}.
\]
Now we fix a constant $M_0>0$, which will be chosen later. We prove that $\mathcal N(t)\le M_0 \eps_0$ for all $t\ge 0$, as long as $\eps_0$ is small enough. 
Let 
\beq\label{def-T}
T_\star=\sup\{T>0:\quad \mathcal N(t)\le M_0 \eps_0\quad \text{for all}\quad 0\le t\le T
\}
\eeq
We shall prove that $T_\star=\infty$ by contradiction argument. Assuming that $T_\star<\infty$, we have 
\[
\mathcal N (T_\star)=\eps_0\quad \text{and}\quad \sup_{0\le t\le T}\mathcal N(t)<\eps_0\quad \text{for all}\quad T<T_\star.
\]
For $t\in (0,T_\star)$, by Theorem \ref{decay-rho} and \ref{thm1}, there exists $C_0,C_1>0$ such that 
\[
\mathcal N(t)\le C_0 \|S\|_{Y_t^N}\le C_0 C_1(\eps_0+\eps_0^2)\]
for $\eps_0$ small enough. 
Let $t\to T_\star$, we have 
\[
M_0\eps_0\le C_0C_1(\eps_0+\eps_0^2)
\]
which is false, as long as $M_0>2C_0C_1$ and $\eps_0$ small.
The proof is complete.

\begin{thebibliography}{10}

\bibitem{chemin}
H.~Bahouri, J.-Y. Chemin, and R.~Danchin.
\newblock {\em Fourier analysis and nonlinear partial differential equations},
  volume 343 of {\em Grundlehren der Mathematischen Wissenschaften [Fundamental
  Principles of Mathematical Sciences]}.
\newblock Springer, Heidelberg, 2011.

\bibitem{book1}
R.~Balescu.
\newblock {\em Statistical mechanics of charged particles}.
\newblock Monographs in Statistical Physics and Thermodynamics, Vol. 4.
  Interscience Publishers John Wiley \& Sons, Ltd.\, London-New York-Sydney,
  1963.

\bibitem{Bardos-Degond}
C.~Bardos and P.~Degond.
\newblock Global existence for the {V}lasov-{P}oisson equation in {$3$} space
  variables with small initial data.
\newblock {\em Ann. Inst. H. Poincar\'{e} Anal. Non Lin\'{e}aire},
  2(2):101--118, 1985.

\bibitem{Landau}
J.~Bedrossian, N.~Masmoudi, and C.~Mouhot.
\newblock Landau damping: paraproducts and {G}evrey regularity.
\newblock {\em Ann. PDE}, 2(1):Art. 4, 71, 2016.

\bibitem{Jacob}
J.~Bedrossian, N.~Masmoudi, and C.~Mouhot.
\newblock Landau damping in finite regularity for unconfined systems with
  screened interactions.
\newblock {\em Comm. Pure Appl. Math.}, 71(3):537--576, 2018.

\bibitem{Sanderson}
T.~J.~M. Boyd and J.~J. Sanderson.
\newblock {\em The physics of plasmas}.
\newblock Cambridge University Press, Cambridge, 2003.

\bibitem{Chavanis}
P.-H. Chavanis.
\newblock Statistical mechanics of violent relaxation in stellar systems.
\newblock In {\em Multiscale problems in science and technology ({D}ubrovnik,
  2000)}, pages 85--116. Springer, Berlin, 2002.

\bibitem{Choi-2D}
S.-H. Choi, S.-Y. Ha, and H.~Lee.
\newblock Dispersion estimates for the two-dimensional {V}lasov-{Y}ukawa system
  with small data.
\newblock {\em J. Differential Equations}, 250(1):515--550, 2011.

\bibitem{global-VP4}
R.~T. Glassey.
\newblock {\em The {C}auchy problem in kinetic theory}.
\newblock Society for Industrial and Applied Mathematics (SIAM), Philadelphia,
  PA, 1996.

\bibitem{HK3}
D.~Han-Kwan.
\newblock Quasineutral limit of the {V}lasov-{P}oisson system with massless
  electrons.
\newblock {\em Comm. Partial Differential Equations}, 36(8):1385--1425, 2011.

\bibitem{HK2}
D.~Han-Kwan and M.~Iacobelli.
\newblock The quasineutral limit of the {V}lasov-{P}oisson equation in
  {W}asserstein metric.
\newblock {\em Commun. Math. Sci.}, 15(2):481--509, 2017.

\bibitem{R2}
D.~Han-Kwan, T.~T. Nguyen, and F.~Rousset.
\newblock Long time estimates for the {V}lasov-{M}axwell system in the
  non-relativistic limit.
\newblock {\em Comm. Math. Phys.}, 363(2):389--434, 2018.

\bibitem{T-R-HK}
D.~Han-Kwan, T.~T. Nguyen, and F.~Rousset.
\newblock Asymptotic stability of equilibria for screened vlasov-poisson
  systems via pointwise dispersive estimates, 2019.

\bibitem{HK1}
D.~Han-Kwan and F.~Rousset.
\newblock Quasineutral limit for {V}lasov-{P}oisson with {P}enrose stable data.
\newblock {\em Ann. Sci. \'{E}c. Norm. Sup\'{e}r. (4)}, 49(6):1445--1495, 2016.

\bibitem{global-VP3}
E.~Horst.
\newblock On the asymptotic growth of the solutions of the {V}lasov-{P}oisson
  system.
\newblock {\em Math. Methods Appl. Sci.}, 16(2):75--86, 1993.

\bibitem{Hwang}
H.~J. Hwang, A.~Rendall, and J.~J.~L. Vel\'{a}zquez.
\newblock Optimal gradient estimates and asymptotic behaviour for the
  {V}lasov-{P}oisson system with small initial data.
\newblock {\em Arch. Ration. Mech. Anal.}, 200(1):313--360, 2011.

\bibitem{global-VP5}
P.-L. Lions and B.~Perthame.
\newblock Propagation of moments and regularity for the {$3$}-dimensional
  {V}lasov-{P}oisson system.
\newblock {\em Invent. Math.}, 105(2):415--430, 1991.

\bibitem{Landau1}
C.~Mouhot and C.~Villani.
\newblock On {L}andau damping.
\newblock {\em Acta Math.}, 207(1):29--201, 2011.

\bibitem{global-VP1}
K.~Pfaffelmoser.
\newblock Global classical solutions of the {V}lasov-{P}oisson system in three
  dimensions for general initial data.
\newblock {\em J. Differential Equations}, 95(2):281--303, 1992.

\bibitem{global-VP2}
J.~Schaeffer.
\newblock Global existence of smooth solutions to the {V}lasov-{P}oisson system
  in three dimensions.
\newblock {\em Comm. Partial Differential Equations}, 16(8-9):1313--1335, 1991.

\bibitem{Smulevici}
J.~Smulevici.
\newblock Small data solutions of the {V}lasov-{P}oisson system and the vector
  field method.
\newblock {\em Ann. PDE}, 2(2):Art. 11, 55, 2016.

\bibitem{wang2018decay}
X.~Wang.
\newblock Decay estimates for the $3d$ relativistic and non-relativistic
  vlasov-poisson systems, 2018.

\end{thebibliography}

\def\cprime{$'$} \def\cprime{$'$}

\end{document}